\documentclass[12pt]{amsart}

\usepackage[english]{babel}

\usepackage{amssymb,amsthm,amsfonts,amsmath}

\usepackage{verbatim}
\usepackage{ulem}
\usepackage{mathtools}

\usepackage{xspace}

\usepackage{url}

\usepackage{graphicx}

\usepackage[dvipsnames]{xcolor}
\DeclareGraphicsExtensions{.pdf,.png,.jpg}

\usepackage{pgfplots}
\usepgfplotslibrary{groupplots}
\pgfplotsset{compat=1.18}
\usepackage{tikz}
\usetikzlibrary{plotmarks}

\usepackage[shortlabels]{enumitem}

\definecolor{darkblue}{rgb}{0.0, 0.0, 0.55}
\usepackage[colorlinks,linkcolor=BrickRed,citecolor=OliveGreen,urlcolor=darkblue,hypertexnames=true,backref=page]{hyperref}
\usepackage[a4paper,margin=2.5cm]{geometry}
\DeclareMathOperator{\Rank}{rank}

\numberwithin{equation}{section}

\newcommand{\Mnc}[1]{ {\mathbb{C}^{#1 \times #1} }}
\newcommand{\Mnmc}[2]{ {\mathbb{C}^{#1 \times #2} }}
\newcommand{\Mnr}[1]{ {\mathbb{R}^{#1 \times #1} }}

\newcommand{\Mnf}[1]{ {\mathbb{F}^{#1 \times #1} }}

\newcommand{\RR}{\mathbb R}
\newcommand{\FF}{\mathbb F}

\newcommand{\CC}{\mathbb C}

\renewcommand{\vec}{\operatorname{vec}}

\newcommand{\df}[1]{{\it{#1}}{\index{#1}}}
\newcommand{\unitL}{{\tt{L}}} 
\newcommand{\unitR}{{\tt{R}}} 
\newcommand{\unitT}{{\tt{T}}}  
\newcommand{\spann}{\operatorname{span}} 
\newcommand{\vn}{n} 
\newcommand{\vm}{m}
\newcommand{\fw}{F} 
\newcommand{\matu}{unit }
\newcommand{\matus }{units}
\newcommand{\JP}{convolution product }
\newcommand{\JPs}{convolution products }
\newcommand{\CP}{\Gamma}
\newcommand{\jcp}{jcp\xspace}
\newcommand{\JCP}{jointly completely positive\xspace}
\newcommand{\Lcpg}{L_{\mathrm{\jcp,gen}}}
\newcommand{\Lcpc}{L_{\mathrm{\jcp,conv}}}
\newcommand{\Lgmv}{L_{\mathrm{GMV}}}

\newcommand{\trace}{\operatorname{trace}}

\makeindex
\makeatletter
\newcommand{\mycontentsbox}{%
{
\parskip=-1.0pt
\newpage\printindex\tableofcontents}}
\def\enddoc@text{\ifx\@empty\@translators \else\@settranslators\fi
\ifx\@empty\addresses \else\@setaddresses\fi
\newpage\mycontentsbox
}
\makeatother

\newtheorem{theorem}{Theorem}[section]
\newtheorem{corollary}[theorem]{Corollary}
\newtheorem{lemma}[theorem]{Lemma}
\newtheorem{proposition}[theorem]{Proposition}
\newtheorem{question}{Question}

\theoremstyle{definition}

\newtheorem{remark}[theorem]{Remark}
\newtheorem{example}[theorem]{Example}

\title{Completely positive matrix products}
\author{Eric Evert, Michael T. Jury, Scott McCullough}
\thanks{MJ partially supported by National Science Foundation Grant DMS-2154494.}  
\date{\today}

\begin{document}

\begin{abstract}
Building on recent works that investigate positivity preserving matrix products and the theory
of jointly completely positive bilinear maps, we examine the class of \JCP (\jcp) matrix products. A bilinear map $\Phi:\Mnc{n} \times \Mnc{n} \to \Mnc{m}$ is a \jcp matrix product if 
 \[
 0 \preceq  \big ( (\Phi(A_{i,j},B_{s,t})_{s,t=0}^{ \ell-1}) \big )_{i,j=0}^{ \ell-1}
 \]
  for all $\ell$ and all positive semidefinite matrices $(A_{i,j})_{i,j=0}^{ \ell-1}, (B_{s,t})_{s,t=0}^{ \ell-1} \in { \Mnc{n\ell}}$. Such products are naturally connected to completely positive maps. In particular, 
    a matrix product is \jcp if and only if the associated linear map out of the tensor product is completely positive, which in turn is equivalent to saying that a naturally associated Choi matrix is positive semidefinite. Similarly, a matrix product is \jcp if and only if it has a Choi-Kraus representation, $\Phi(A,B) = \sum V_j^* (A \otimes B) V_j$. We use the Choi-Kraus representation of \jcp matrix products to study various basic properties, including positivity lower bounds, commutativity, units, causality, and separability. As examples, we apply our results to the Schur (Hadamard) product and the convolution product. 
\end{abstract}

\maketitle

\maketitle

\noindent \textit{Keywords}: Matrix products,  jointly completely positive bilinear maps,  positivity preserving, completely positive, Choi-Kraus representation, Schur product, convolution product, Oppenheim-Schur bounds.

\noindent \textit{MSC 2020}: Primary: 15B48, 47A20. Secondary: 46L07, 15A42.

\section{Introduction}

Positivity is a central theme in functional analysis, matrix analysis, and operator theory. For example,
it plays a key role in matrix decompositions \cite{BP,LS}, $C^*$-algebras \cite{Arv,Dav}, 
quantum information theory, and convex duality \cite{Alf}. Recently, positive-preserving matrix products have been explored in works such as \cite{lower} and \cite{dominiqueOppenheim}. Here, for positive integers $\vn$ and $\vm,$  a \df{matrix product} is a bilinear map $\Phi:\Mnc{n} \times \Mnc{n} \to \Mnc{m},$
where \df{$\Mnc{n}$} denotes the $n\times n$ matrices $A=(A_{jk})_{j,k=0}^{n-1}$ over $\CC.$ { Following \cite{dominiqueOppenheim},} 
a matrix product $\Phi$ is \df{positivity-preserving} if $\Phi(A,B) \succeq 0$ whenever $A,B \succeq 0,$ where $T\succeq 0$ \index{$\succeq$}
 indicates that the square matrix $T$ is \df{positive semidefinite} (\df{psd}).

Classical examples of positivity-preserving  products include the \df{Schur (Hadamard) product} $S:\Mnc{n} \times \Mnc{n} \to \Mnc{n}$ defined by $S(A,B)_{ij} = A_{i,j} B_{i,j}$, and the \df{\JP} $\CP:\Mnc{n} \times \Mnc{n} \to \Mnc{n}$ defined by \index{$\CP$} \index{$S$}
\[
\CP (A, B)_{ij} = \sum_{m=0}^i \sum_{n=0}^j A_{m,n} B_{i-m,j-n}.
\]
Both of these products have received extensive attention since their introduction and both play a role in interpolation problems of Pick type; e.g., see \cite{Horn,Wagner,Jury,DMM,AM}. 

Recently, the article \cite{dominiqueOppenheim} illustrated that the positivity-preserving property of these products can be quantified via Oppenheim-Schur type lower bounds \cite{FallatJohnson, Oppenheim, Schur}. Indeed, their work derives impressive lower bounds for the determinants of the Schur and \JP of positive-definite matrices. The article \cite{dominiqueOppenheim} additionally extends these lower bounds to a related class of \df{causal} matrix products, which intuitively capture a forward flow of information.  

A unifying feature of the Schur and \JPs is that both are representable in a \df{Choi-Kraus form} $V^* (A \otimes B) V$ for appropriately chosen matrices $V \in \mathbb{C}^{n^2 \times n}$. Here \df{$\otimes$} is the tensor product of operators on Hilbert space and may be 
identified with the Kronecker product if desired. This Choi-Kraus form provides a unifying qualitative explanation for why these products all preserve positivity. It also says that both products arise from completely positive maps.

Recall that, for a positive integer $\ell,$  a linear map $\Psi:\Mnc{n} \to \Mnc{m}$ is 
\df{$\ell$-positive} if for all positive semidefinite block matrices $0 \preceq A = (A_{ij})_{ij} \in \Mnc{n\ell}$  the block \df{ampliation}
\[
\Psi_\ell((A_{ij})_{ij}):= (\Psi(A_{ij}))_{ij} \in \Mnc{m\ell} \cong M_\ell(M_m)
\]
is positive semidefinite. The map $\Psi$ is \df{\JCP} (\df{\jcp}) if $\Psi$ is $\ell$-positive for all $\ell \in \mathbb{N}$. Such maps play a central role in the theory of operator spaces and systems and $C^*$-algebra, see e.g. \cite{PaulsenBook}.

As shown in the work of Choi \cite{Choi},  a linear map $\Psi:\Mnc{n} \to \Mnc{m}$ is completely positive if and only if it is $n$-positive. In fact, one need only check that its \df{Choi matrix}  \index{$C_\Phi$}
\[
  C_\Psi:=\Psi_n((E_{ij})_{i,j=0}^{n-1})
\]
 is positive semidefinite. Here the \df{$E_{ij}$} denote the standard \df{matrix units} in $\Mnc{n}$. Based on a natural decomposition of this Choi matrix, Choi further showed $\Phi$ is completely positive if and only if there exist a finite collection of matrices 
 $\{V_j\}_{j=1}^k \in \mathbb{C}^{n \times m}$   such that 
\[
 \Psi(A) = \sum_{j=1}^k V_j^* A V_j
\]
 for all matrices $A \in \Mnc{n}$. That is, a linear map $\Psi$ is completely positive if and only if it has a \df{Choi-Kraus representation}.

 The article \cite{KPTT} studies jointly completely positive 
 bilinear maps. Specializing to the case of 
 matrix algebras, we say a bilinear map 
 $\Phi:\Mnc{n} \times \Mnc{n} \to \Mnc{m}$ \sout{ such that}
 is a \df{jointly completely positive (jcp) matrix product}
 if \index{jcp}
 \[
 0 \preceq  \big ( (\Phi(A_{i,j},B_{s,t})_{s,t=0}^{ \ell-1}) \big )_{i,j=0}^{ \ell-1}
 \]
 for all $\ell$ and all positive semidefinite matrices $(A_{i,j})_{i,j=0}^{ \ell-1}, (B_{s,t})_{s,t=0}^{ \ell-1} \in { \Mnc{n\ell}}$. 
  It is evident that jcp matrix products
 are positivity preserving and this class implicitly appears in \cite{lower}.  See
 Remark~\ref{r:implicit-in-GMV}. 
 
  Theorem~\ref{thm:CPPForms}
  expands on Theorem 5.8 of \cite[Theorem 5.8]{KPTT} by including the Choi matrix that is available in the finite dimensional setting.  That is, a matrix product is \jcp if and only if a corresponding Choi matrix is positive semidefinite  if and only if the matrix product has a Choi-Kraus representation. We leverage this Choi-Kraus form to study basic algebraic properties, such as commutativity and units, and to provide alternate proofs of several lower bounds found in \cite{dominiqueOppenheim}. 

\subsection{Guide to the article}
\label{sec:guide}
 In Section \ref{s:cppp} Choi's theorem is naturally extended  to the setting of \jcp products. 
 Theorem~\ref{thm:CPPForms} records basic facts
 about \jcp matrix products that flow from completely positivity. See also \cite[Theorem 5.8]{KPTT}. A matrix product $\Phi:\Mnc{n} \times \Mnc{n} \to \Mnc{m}$ is \jcp if and only if a natural Choi matrix is positive semidefinite, which  in turn corresponds to $\Phi$ having a Choi-Kraus representation. In Section \ref{sec:LowerBounds}, we derive lower bounds for the positivity of a \jcp product $\Phi(A,B)$ based on the Choi-Kraus representation of $\Phi$. In particular, Proposition \ref{p:generic-lower} provides an easily computable, although potentially far from optimal, lower bound on the determinant of an arbitrary \jcp product.
 { Meanwhile Theorems \ref{p:DMV-Prod+lower} and \ref{t:lbJ} use the \jcp Choi-Kraus representation framework to offer modest refinements of \cite[Theorem A]{lower} and \cite[Theorem 2.1]{lower}, respectively.}  
Section \ref{sec:Causal} studies causal \jcp products. An advantage of the \jcp framework here is that it provides a natural description of \df{left causal} and \df{right causal} products. The main result of this section is Theorem \ref{t:causal}.  It  shows that a \jcp product $\Phi$ is left causal if and only if it has a Choi-Kraus representation $\Phi(A,B) = \sum V_j^* (A \otimes B) V_j$ where each $V_j \in \mathbb{C}^{n^2 \times n}$ is \df{left upper triangular}.  In Section \ref{sec:BasicAlgebra} we study basic algebraic properties of \jcp products such as commutativity, associativity, and the existence of units. Section \ref{sec:separable} introduces \df{separable \jcp products}, i.e., \jcp products whose Choi matrix is separable in a natural sense. It begins by recording, as Proposition \ref{prop:SeparableKrausForm},
the anticipated fact that a \jcp product is separable if and only if it has a Choi-Kraus representation where each $V_j$ is rank one.
Using this fact we proceed to
show that separability imposes strong restrictions on  units of separable \jcp products. The article ends with Section \ref{sec:NumExp}. It presents numerical experiments that illustrate the quality of the lower bounds obtained in Section \ref{sec:LowerBounds}.

 The article uses some standard conventions that we spell out here.  Some results are basis
 dependent. In these cases we fix a basis $\{e_0,\dots, e_{n-1}\}$ of $\CC^n$ and order the orthonormal
 basis $\{e_i\otimes e_s: 0\le i,s<n\}$ for $\CC^n\otimes \CC^n\cong \CC^{n^2}$  using the lexicographic order with $0>1>\dots >n-1.$ Thus, $e_0\otimes e_j$
 comes before any $e_1\otimes e_k.$ In the case of $n=2,$ the basis is, in order, $e_0\otimes e_0, e_0\otimes e_1, 
 e_1\otimes e_0, e_1\otimes e_1.$  This choice is consistent with the usual convention for the Kronecker product;
 that is, if $A,B\in \Mnc{n},$ then $A\otimes B$ with respect to this basis is the standard Kronecker product.

 The \df{vectorization of a matrix} $A\in \Mnc{n}$ in terms of its columns is
\[
 \vec(A) = \sum_{i,s} e_i \otimes Ae_s.
\]
Thus, for instance, for $A=(a_{u,v})_{u,v=0}^1 \in \Mnc{2},$
\[
 \vec(A) = \begin{pmatrix} a_{00} & a_{10} & a_{01} & a_{11}\end{pmatrix}^\top,
\]
 where \df{${}^\top$} denotes transpose.

\section{Complete Positivity Products}
\label{s:cppp}

A  bilinear map $\Phi: \Mnc{n} \times \Mnc{n} \mapsto \Mnc{m}$ is \df{positivity preserving} if $\Phi(A,B) \succeq 0$
for all positive semidefinite $A,B \in \Mnc{n}$.
 It is \df{\JCP} (\df{\jcp}) if for each $\ell \in \mathbb{N}$ and for all positive semidefinite matrices $(A_{ij})_{i,j=0}^{ \ell-1}, (B_{s,t})_{s,t=0}^{ \ell-1} \in { \Mnc{n\ell}} \cong M_\ell(M_n)$, 
\begin{equation}
    \label{d:cpp}
    \begin{split}
0  & \preceq  \big ( (\Phi(A_{i,j},B_{s,t})_{s,t=0}^{ \ell-1}) \big )_{i,j=0}^{ \ell-1}
\\  &  =\sum_{i,j,s,t=0}^{{ \ell-1}}  E_{i,j} \otimes E_{s,t} \otimes  \Phi(A_{i,j},B_{s,t}) \in { \Mnc{\ell^2 n} \cong M_{\ell} (M_\ell( M_n))},
\end{split}
\end{equation}
 where \df{$E_{i,j}$} for $0\le i,j\le { \ell-1}$ are the \df{matrix units} for ${ \Mnc{\ell}}$ and \df{$M_n(\mathcal{A})$} denotes the $n\times n$
 matrices with entries from $\mathcal{A}.$ Clearly, a \jcp product is positivity preserving.

The identification of a bilinear map $\Phi:\Mnc{n}\times \Mnc{n}\to \Mnc{m}$ 
with its lift to a linear map $\Psi$ from the tensor product $\Mnc{n}\otimes \Mnc{n}$ into $\Mnc{m}$
naturally leads to the following characterization of \jcp products, which  essentially appears as 
Theorem 5.8 of \cite{KPTT}.

\begin{theorem}
\label{thm:CPPForms}
For a bilinear map;  $\Phi: \Mnc{n} \times \Mnc{n} \mapsto \Mnc{m},$ the following are equivalent.  \index{$C_\Phi$}
\begin{enumerate}[(i)] \itemsep=6pt 
    \item \label{it:CPPDef} The map $\Phi$  is \jcp. 
    \item \label{it:CPPChoi} The matrix 
\[
  \begin{split}
      C_\Phi   & := 
   ((\Phi(E_{i,j},E_{s,t})_{i,j=0}^{n-1}))_{s,t=0}^{n-1}
  \\ &  = \sum_{i,j,s,t=0}^{n-1}  E_{i,j} \otimes E_{s,t} \otimes  \Phi(E_{i,j},E_{s,t}) \in \Mnc{n} \otimes \Mnc{n} \otimes \Mnc{m},
 \end{split}
\]
  is psd.  
    \item \label{it:CPPKraus} There exist an $N \leq mn^2$ and matrices $V_1, \dots V_N \in \Mnmc{n^2}{m}$ such that 
    \[
     \Phi(A,B) = \sum_{k=1}^N V_k^* (A \otimes B) V_k.
    \]
    \item \label{it:CPPMJ} The canonical linear map $\Psi: \Mnc{n}\otimes \Mnc{n} \to \Mnc{m}$ on the tensor product induced by $\Phi,$ so that $\Psi(A\otimes B) =\Phi(A,B),$ is completely positive.
    \item \label{it:npos}  The positivity of equation~\eqref{d:cpp} holds for $\ell=n.$ 
\end{enumerate}
\end{theorem}

\begin{proof}
    First note that if \ref{it:CPPDef} holds, then  $\Phi(A,B) = C_\Phi$
     for the choices $A=(E_{i,j})_{i,j=1}^n \succeq 0$ and 
     $B=(E_{s,t})_{s,t=1}^n.$  Hence $C_\Phi \succeq 0.$ Now suppose \ref{it:CPPChoi} holds so that  $C_\Phi \succeq 0.$ Let $N \leq mn^2$ denote the rank of $C_\Phi$. There exist vectors $\{w_k\}_{k=1}^N \in \mathbb{C}^{mn^2}$ such that 
    \[
    C_\Phi = \sum_{k=1}^{N} w_k w_k^*. 
    \]
    Identify $\mathbb{C}^{mn^2} \cong   (\mathbb{C}^n \otimes \mathbb{C}^n)\otimes \mathbb{C}^m$ and write $w_k\cong (w_k(i,s))_{i,s=0}^{n-1}$ with respect to this identification, so $w_k(i,s) \in \mathbb{C}^m$ for each $k,i,s$. Define
    \[
    V_k = (w_k(i,s)^*)_{i,s=0}^{n-1} 
     =\sum_{i,s=0}^{n-1} (e_i\otimes e_s) \, w_k(i,s)^* \in \Mnmc{n^2}{m},
    \]
    where $\{e_0,\dots,e_{n-1}\}$ is the standard
     orthonormal basis for $\CC^n.$
    Thus $V_k$ is  an $n^2 \times  m$ matrix with rows $w_k(i,s)^*$. 
    From 
\begin{align*}
    V_k^* (E_{i,j} \otimes E_{s,t}) V_k &= V_k^* (e_i e_j^* \otimes e_s e_t^*) V_k \\
    & = V_k^* (e_i \otimes e_s) (e_j^* \otimes e_t^*) V_k \\
    & = w_k(i,s) w_k (j,t)^*,
\end{align*}
    it follows that $\sum_k V_k^* (E_{i,j} \otimes E_{s,t}) V_k = \sum w_k(i,s) w_k (j,t)^* = ((C_\Phi)_{i,j})_{s,t} = \Phi(E_{i,j},E_{s,t})$. The  bilinearity of $\Phi$ now gives
    \[
    \sum_{k=1}^N V_k^* (A \otimes B) V_k = \Phi(A,B)
    \]
    for any matrices $A,B$. {Thus item~\ref{it:CPPChoi} implies item~\ref{it:CPPKraus}. It is straightforward to check that item~\ref{it:CPPKraus} implies item~\ref{it:CPPDef}.}

Since the map $\Phi$ is bilinear, it factors through the tensor product $\Mnc{n}\otimes \Mnc{n};$ that is, there
exists a linear map $\Psi:\Mnc{n}\otimes \Mnc{n}\to \Mnc{m}$ such that $\Phi(A,B)=\Psi(A\otimes B).$  The Choi matrix $C_\Psi$ for $\Psi$  is
\[
\begin{split}
 C_\Psi  & = \begin{pmatrix} \begin{pmatrix} \Psi(E_{i,j}\otimes E_{s,t}) \end{pmatrix}_{i,j} \end{pmatrix}_{s,t}
  =\sum_{i,j,s,t}   \left ( E_{i,j}\otimes E_{s,t} \right ) \otimes \Psi(E_{i,j}\otimes E_{s,t})
 \\ & = \begin{pmatrix} \begin{pmatrix} \Phi(E_{i,j},E_{s,t}) \end{pmatrix}_{i,j} \end{pmatrix}_{s,t} = C_\Phi,
\end{split}
\]
 from which it follows that item~\ref{it:CPPMJ} is equivalent to the other items since the complete positivity of $\Psi$ is equivalent to the positivity of its Choi matrix.

It is immediate that item~\ref{it:CPPDef} implies 
item~\ref{it:npos} and in turn item~\ref{it:npos} implies
item~\ref{it:CPPChoi}. The proof is now complete.
\end{proof}

\begin{remark}
\label{r:CPPP-form-terminology}
  The representation of $\Phi$ in item~\ref{it:CPPKraus} is
  known as a \df{Choi-Kraus form}. The matrices $V_k$ are
  the \df{coefficients} of the representation. The minimum number
  of $V_k$ needed is the \df{Choi-Kraus rank} of $\Phi.$

  The matrix  $C_\Phi$ in item~\ref{it:CPPChoi} is the \df{Choi matrix} of $\Psi$ in item~\ref{it:CPPMJ},
  where $C_\Psi=C_\Phi.$ We will also refer to this matrix and the Choi matrix of $\Phi.$
  As a consequence of the proof of the theorem, the rank of $C_\Psi$ ($C_\Phi$) is equal to the Choi-Kraus rank of $\Psi$
  ($\Phi$).
  \qed
\end{remark}

\begin{remark}\rm
 \label{r:Kraus-to-Choi}
 The definition of a \JCP product
 and Theorem~\ref{thm:CPPForms} extend to the setting
 of bilinear maps $\Phi:\Mnc{n_1}\times \Mnc{n_2}\to \Mnc{m}.$
 We eschew this extra generality here.
 \qed
\end{remark}

\begin{remark}\rm
 \label{r:implicit-in-GMV}
   A bit of unravelling reveals that the matrix $Y$ appearing
   in \cite[Definition~2.1(b)]{lower} is a Choi matrix
   and hence the class $\mathbf{Prod_+}$ appearing there is precisely the class of 
   \jcp products.
   \qed
\end{remark}

\subsection{Examples of Choi-Kraus decompositions} As examples of Theorem \ref{thm:CPPForms}, we present the Choi-Kraus decompositions of the Schur and \JPs.

\subsubsection{The Schur product map}
\label{sss:Schur-prod}
 It is well known that the Schur product $S:\Mnc{n}\times \Mnc{n}\to \Mnc{n}$ defined by 
\[
 S(A,B)_{i,j} =A_{i,j}\ B_{i,j},  
\]
  is \jcp; that is, the map it determines on $\Mnc{n}\otimes \Mnc{n}$ is \jcp. The map $V:\CC^n\to \CC^n\times \CC^n$ determined by $Ve_j =e_j\otimes e_j$ gives,
\[
  S(A,B) = V^* (A\otimes B) V.
\]
 Thus $S$ satisfies item~\ref{it:CPPKraus} of Theorem~\ref{thm:CPPForms}.
 In particular, the Choi-Kraus rank of $S$ is one.

 To compute the Choi matrix for $S,$ first observe that
\[
 S(E_{i,j},E_{s,t}) = \delta_{(i,j),(s,t)} \ E_{i,j}.
\]
Hence
\[
\begin{split}
 C_S & = \sum_{i,j,s,t=0}^{n-1} S(E_{i,j},E_{s,t}) \ \otimes \left ( E_{i,j}\otimes E_{s,t} \right )
 \\ & = \sum_{i,j} E_{i,j} \otimes  \left ( E_{i,j}\otimes E_{i,j} \right ).
\end{split}
\]

\subsubsection{The \JP map}
 \label{sss:convo-product}
Recall the \JP $\CP:\Mnc{n}\times \Mnc{n}\to \Mnc{n}$ defined by
\[
\CP(A, B)_{i,j} = \sum_{m=0}^i \sum_{n=0}^j A_{m,n} B_{i-m,j-n}.
\]
Let $\{e_j\}_{j=0}^{n-1}$ denote the standard basis in $\mathbb{C}^n$ and  $V:\CC^n\to \CC^n\otimes \CC^n$ denote the map defined by
 $Ve_j =  \sum_{s+u=j}e_s\otimes e_u$. A computation reveals the Choi-Kraus form 
\begin{equation}
    \label{e:Choi-for-Gat}
   \CP(A,B) = V^* (A\otimes B) V,
\end{equation}
 for $\Gamma.$
 In particular, the  Choi-Kraus rank of $\CP$ is one.  Moreover, the Choi matrix for $\CP$ is
 \[
  C_{\CP} = \sum \{ (e_{i+s}\otimes e_{j+t}^*)  \otimes   (e_i\otimes e_s) (e_i\otimes e_s)^* \mid  i+s, j+t\le n\}.
\]

As a final remark, the Choi-Kraus coefficient in the Choi-Kraus representation of $\CP$ in equation~\eqref{e:Choi-for-Gat} has the alternate description
\[
  V^* = \begin{pmatrix} I &  \fw &  \fw^2 & \dots &   \fw^{N-1}\end{pmatrix} = \sum_{j=0}^{n-1}  e_j^* \otimes \fw^{j} = \sum_{j=0}^{n-1} \fw^{j} \otimes e_j^*,
\]
  where $\fw \in \Mnc{n}$ is the \df{forward shift}, $\fw e_j=e_{j+1}$ for $0\le j< n-1$ and $\fw e_{n-1}=0.$ Equivalently, 
  \begin{equation}
    \label{e:VforGat}
  V= \sum_{j=0}^{n-1}  e_j \otimes \fw^{*j} =  \sum_{j=0}^{n-1} \fw^{*j} \otimes e_j.
  \end{equation}
  The equality of the latter two expressions 
  encodes the commutativity of the convolution product. See the upcoming Theorem \ref{thm:CommutativeChoiRep}.

\section{Lower bounds for \jcp products}
\label{sec:LowerBounds}

In this section we present lower bounds for the positivity of a \jcp image of positive definite matrices. Numerical experiments that illustrate the quality of the bounds are provided in Section \ref{sec:NumExp}. 

We first provide a naive lower bound that is quickly computable from a Choi-Kraus representation. Let $\lambda_j(A)$ denote the $j$-th largest eigenvalue value of a matrix $A$ and  $\sigma_j(A)$ denote the $j$-th largest singular value of $A$, including $0$ as a possibility. The proofs of the following generic lower estimates are routine.

\begin{proposition}
\label{p:generic-lower}
 Let $\Phi:\Mnc{n} \times \Mnc{n} \to\Mnc{m}$ be a \jcp product with Choi-Kraus representation 
 \[
\Phi(A,B) = \sum_{j=1}^\ell V_j^* (A \otimes B) V_j,
 \]
 and set $V^* := [V_1^* \dots V_\ell^*] \in \mathbb{C}^{m \times n^2\ell}$. If $0 \preceq A,B \in \Mnc{n}$ and $m \leq \ell n^2$,  then
 \[
 \lambda_m(\Phi(A,B)) \geq \lambda_n(A) \lambda_n(B) \sigma_m(V)^2
 \]
 and 
 \[
 \det(\Phi(A,B)) \geq \prod_{j=0}^{m-1} \Big(\lambda_{\ell n^2 - j} (I_\ell \otimes A \otimes B) \sigma_{m-j} (V)^2\Big).
 \]
  Moreover both estimates are sharp. 
\end{proposition}

\begin{remark}\rm
  While the estimates in Proposition~\ref{p:generic-lower} are sharp,
  they are generically far from optimal and {typically} do not compete
  with the remarkable lower bounds given in the special cases that appear 
  in \cite{lower}. 
  \qed
\end{remark}

The complete positivity machinery provides a modestly streamlined
proof  and modest improvement of the inequality portion of \cite[Theorem~A]{lower}.
It will resonate with those familiar with complete positivity. 

\begin{theorem}
    \label{p:DMV-Prod+lower}
    If $\Phi:\Mnc{n}\times \Mnc{n}\to \Mnc{m}$ is a \jcp product with Choi-Kraus form
    as in item~\ref{it:CPPKraus} of Theorem~\ref{thm:CPPForms} and $A,B\in \Mnc{n}$ are non-zero, then
\[
    \Phi(AA^*,BB^*) \succeq  \frac{1}{u} \sum_k R_kR_k^*,
\]
 where the $R_k$ are the column vectors,
\[
  R_k = V_k^* \vec(BA^{\top}), 
\]
and  $u=\min\{\Rank AA^*, \, \Rank BB^*\}.$ 
\end{theorem}

The proof yields a bit more than claimed.

\begin{proof}
    In view of the Choi-Kraus form, it suffices to prove
\[
   AA^* \, \otimes \, BB^* \succeq 
   \frac{1}{u} \vec(BA^\top) \vec(BA^\top)^*.
\]
To this end, let $A,B\in \Mnc{n}$ be given and let 
$e_0,\dots,e_{n-1}$ denote an orthonormal basis for $\CC^n.$
Letting 
\[
 \mathbf{e} =  \sum_{j=0}^{n-1} e_j \otimes e_j,
\]
 observe that 
\begin{equation}
    \label{e:e-in-action}
  (A\otimes B) \mathbf{e} = \vec(BA^\top),
\end{equation}
 where $A^\top$ is computed with respect to the given orthonormal basis. 

 Let $r$ and $s$ denote the ranks of $A$ and $B.$ Since $A\ne 0\ne B,$
 both  $r$ and $s$ are positive.  Let $P$ and
 $Q$ denote the projections onto 
 the orthogonal complements of the kernels
 of $A$ and $B$ respectively. Thus $P$ and $Q$
 are the projections onto the ranges of $A^*$ and $B^*$
 and moreover, $\trace P = r$ and $\trace Q=s.$
 In particular,  using $0\le P_{jj}, \, Q_{jj} \le 1,$
\[
  t = \sum_{j=1}^n P_{jj} Q_{jj}  \le \min\{r,s\}
\]
 and 
\[
 (A\otimes B) (P\otimes Q) (e_j\otimes e_j) 
  = (A\otimes B) (e_j \otimes e_j)
\]
 so that
\begin{equation}
    \label{e:ef-in-action}
(A\otimes B) \, \mathbf{g}  = 
 (A\otimes B) \, \mathbf{e},
\end{equation}
 where $\mathbf{g}= (P\otimes Q) \, \mathbf{e}.$
Since 
\[
  \|\mathbf{g}\|^2 
   = \sum_j (e_j^* P^*P e_j)\, (e_j^* Q^*Q e_j)
    =\sum_j P_{jj}Q_{jj} \le t,
\]
it follows that 
\begin{equation}
    \label{e:f-in-action}
0 \preceq G:=\mathbf{g}\, \mathbf{g}^*/t\preceq I.
\end{equation}
Hence, combining equations~\eqref{e:e-in-action}, \eqref{e:ef-in-action} and \eqref{e:f-in-action},
\[
\begin{split}
 AA^*\otimes BB^* & = (A\otimes B) \, (A\otimes B)^*
   \succeq  (A\otimes B) G (A\otimes B)^*
   \\[3pt] & = \frac{1}{t} (A\otimes B) \mathbf{g} \,
   \mathbf{g}^* (A\otimes B)^* 
    = \frac{1}{t} (A\otimes B) \mathbf{e}\, \mathbf{e}^*
   (A\otimes B)^* 
 \\[3pt] & = \frac{1}{t} \vec(BA^\top) \, \vec((BA^\top)^*),
 \end{split}
\]
 which proves a bit more than claimed.
 \end{proof}

In \cite[Theorem 2.1]{lower} the authors report on
their discovery of a remarkably clean lower bound
for the determinant of the {convolution} product. 
The following result offers a small improvement.

\begin{theorem}
    \label{t:lbJ}
       For $n\ge 2$ and  psd matrices $P,Q\in M_n(\CC),$
 \[
  \det (\CP(P,Q)) \ge \left ( p_{00}^{\frac{n}{n-1}} \det(Q)^{\frac{1}{n-1}} + q_{00}^{\frac{n}{n-1}} \det(P)^{\frac{1}{n-1}} \right )^{n-1},
 \]
  where $\CP$ is the convolution product.  Moreover, equality holds when $n=2.$
\end{theorem}

\begin{remark}\rm
\label{r:lbJ}
Using 
\[
\left ( p_{00}^{\frac{n}{n-1}} \det(Q)^{\frac{1}{n-1}} + q_{00}^{\frac{n}{n-1}} \det(P)^{\frac{1}{n-1}} \right )^{n-1}
  \ge p_{00}^n \det(Q) + q_{00}^n \det(P)
\]
recovers  \cite[Theorem 2.1]{lower}. Note, when $n=2$ the inequality is an equality since $n-1=1.$
\qed
\end{remark}

\begin{proof}
The proof  
uses Minkowski's inequality for psd
 $k\times k$ matrices $X$ and $Y,$
\begin{equation}
 \label{e:minkowski}
 \det(X+Y) \ge (\det(X)^{\frac{1}{k}} + \det(Y)^{\frac{1}{k}})^k.
\end{equation}
It will also use the elementary fact based upon the Schur complement that,
for a positive integer $m,$ 
a non-zero scalar $a,$ a vector $b\in \CC^m$ and an $m\times m$ matrix $d,$
\begin{equation}
    \label{e:schur-det} \det \begin{pmatrix} a & b^*  \\ b & d \end{pmatrix} = a \det(d- \frac{bb^*}{a}).
\end{equation}

Let  $\{e_0,e_1,\dots,e_{n-1}\}$ denote 
the (fixed) orthonormal basis for $\CC^n.$
Suppose $P,Q$ are psd and $n\times n.$  Write $P=A^*A$ and $Q=B^*B$ where
$A=(a_{ij})$ and $B=(b_{ij})$ are upper triangular.   Express
$\CP$ in the Choi-Kraus form from subsection~\ref{sss:convo-product}
\[
 \CP(P,Q) = V^*(P\otimes Q) V = V^* (A^* \otimes B^*) (A\otimes B) V,
\]
 where
 \[
  V= \sum_{j=0}^{n-1} e_j \otimes F^{*j}
 \]
  and $F$ the forward shift as in equation~\eqref{e:VforGat}.
 Let $R_k=e_k e_0^*$ for $k=1,2,\dots,n-1$ and let
$R$ denote the block diagonal 
$n\times n$ matrix with  $0$-th diagonal entry $I$ and  $k$-th
diagonal entry $R_k$ for $1\le k<n.$ In particular, $R^*R \preceq I$ and hence,
\begin{equation}
    \label{e:J1}
 \CP(P,Q) \succeq  V^* (A^* \otimes B^*) R^* \, R (A\otimes B) V.
\end{equation}

 Since $A$ is upper triangular, 
\[
  R (A\otimes B)V  = \begin{pmatrix} \sum_{j=0}^{n-1} a_{0j} R_0 B \fw^{*j} \\[3pt] \sum_{j=1}^{n-1} a_{1j} R_1 B \fw^{*j} 
   \\[3pt]  \vdots  \\[3pt]  a_{n-1,n-1}  R_{n-1} B\fw^{*(n-1)}
\end{pmatrix}.
\] 
 Let $D_k$ denote the $k$-th block entry of $R(A\otimes B)V$.
 Thus for $1\leq k\leq n-1,$ 
\[
 D_k = R_k \left (\sum_{j=k}^{n-1} a_{kj}  B\fw^{*j} \right ) = e_k \sum_{j=k}^{n-1}  a_{kj} b_{0,j-k} e_j^* \in \Mnc{n}.
\]
 In particular, for $1\le \ell\ne k \le n-1,$ 
\[
 D_\ell^*  D_k  =0. 
\]
Consequently, 
\begin{equation}
    \label{e:J1.5}
V^* (A^* \otimes B^*) R^* \,  R  (A\otimes B) V = D_0^* D_0 + \sum_{k=1}^{n-1} D_k^* D_k = D_0^* D_0 + E^* E,
\end{equation}
 where
\[
  E=\sum_{k=1}^{n-1} D_k. 
\]
 Observe that $E$ is upper triangular with first row $0$ and diagonal entries $a_{kk}b_{00}$ for $1\le k<n.$
  Thus, 
\[
 D_0^* D_0 = \begin{pmatrix} \alpha & \beta \\ \beta^* & \delta \end{pmatrix} 
\]
and 
\[
 E^*E =  \begin{pmatrix} 0 & 0 \\ 0 & G^*G \end{pmatrix},
\]
 where $G$ is upper triangular with diagonal entries $a_{kk}b_{00}$. 
 From the Schur complement formula of equation~\eqref{e:schur-det} for the determinant of a block matrix
\begin{equation}
    \label{e:J2}
 \det(D_0^*D_0 + E^*E)  = \alpha \det( \delta + G^*G - \frac{ \beta^* \beta}{\alpha})
  = \alpha \det ( (\delta- \frac{ \beta^* \beta}{\alpha}) + G^*G).
\end{equation}
Applying the Minkowski inequality \eqref{e:minkowski} with $k=n-1$ and $X=\delta- \frac{ \beta^* \beta}{\alpha}$
and $Y=G^*G$ from equation~\eqref{e:J2} and then equation~\eqref{e:schur-det}
 to $D_0^*D_0,$
\begin{equation}
    \label{e:J3}
    \begin{split}
  \det(D_0^* D_0 + E^*E)   & \ge  \left [ \left (\alpha \det ( \delta-\frac{ \beta^* \beta}{\alpha}) \right )^{\frac{1}{k}}  +  \left (\alpha \det(G^*G)\right )^{\frac{1}{k}} \right ]^k
  \\[3pt] & = \left [ \det(D_0^* D_0)^{\frac{1}{k}}  +  \left (\alpha \det(G^*G)\right )^{\frac{1}{k}} \right ]^k.
\end{split}
\end{equation}

 Since $B$ is upper triangular, so is 
 $D_0.$ Moreover,  the  diagonal entries of $D_0$ are  $a_{00} b_{kk}.$
  Hence
\begin{equation}
    \label{e:J4}
   \det(D_0^*D_0) = \big \vert \det(D_0)\big\vert^2 = |a_{00}|^{2n} \det(B^*B) = p_{00}^n \det(Q).
\end{equation}
Since $G$ is also upper triangular with diagonal entries $b_{00} a_{kk}$ (for $1\le k \le n-1$) 
  and $\alpha = |b_{00}|^2 |a_{00}|^2 = p_{00} q_{00},$ 
\begin{equation}
    \label{e:J5}
\begin{split}
  \alpha \det(G^*G)&  = \big \vert a_{00} b_{00}  \det(G) \big  \vert^2 = |b_{00}|^{2n}  | a_{00} a_{11}\cdots a_{n-1,n-1} |^2
  \\[3pt] &  = q_{00}^n \,  |\det(A)|^2 =  q_{00}^n \det(P).
\end{split}
\end{equation}
 Combining equations~\eqref{e:J1}, \eqref{e:J1.5}, \eqref{e:J3}, \eqref{e:J4} and \eqref{e:J5} gives, 
 \[
\begin{split}
 \det(\CP(P,Q)) & \ge  \det(D_0^* D_0 + E^*E) 
 \\[3pt] & \ge \left (\det(D_0^*D_0)^{\frac{1}{n-1}} + 
  \left (\alpha \det(G^*G)\right )^{\frac{1}{n-1}} \right )^{n-1}
 \\[3pt] &  = \left ( p_{00}^{\frac{n}{n-1}} \det(Q)^{\frac{1}{n-1}} + q_{00}^{\frac{n}{n-1}} \det(P)^{\frac{1}{n-1}} \right )^{n-1},
 \end{split}
\]
  completing the proof of the inequality of Theorem~\ref{t:lbJ}.
  
 That equality holds when $n=2$ can be checked by direct computation. Alternatively, 
 observe that when $n=2,$ because the last row of $V$ is $0,$ equality holds in
 the inequality of equation~\eqref{e:J1}; and since $\alpha \delta - \beta^*\beta$ and $G^*G$
 are non-negative real numbers, equality also holds in the inequality of equation~\eqref{e:J3}.
\end{proof}

\begin{remark}
(1) We discerned no clean statement in the further improvements of the inequality of Theorem~\ref{t:lbJ} based upon 
 more refined choices for the $R_j$ in the proof given here.

\bigskip

(2)  The proof generalizes with the $V$ for the \JP
 replaced by $V$ of the form,
\begin{equation}
  \label{e:altV}
  V=\sum_{i+j\le k} c_{ijk} (e_i\otimes e_j) e_k^*.  
\end{equation}
 Equivalently, with $V$ that  satisfy
\begin{equation}
 \label{e:perm-causal-alt}
   \langle Ve_k,e_i\otimes e_j \rangle =0 
\end{equation}
 for $i+j>k.$ In the case that 
  $c_{ijk}=1$ when $i+j=k$ the conclusion of Theorem~\ref{t:lbJ} is unchanged.  For general $c_{ijk},$ the proof
 goes through with a conclusion that includes the terms $c_{ijk}$ for $i+j=k.$ The
 details are left to the curious reader.

\bigskip

(3) The $V$ of equation~\eqref{e:altV} offers another take on causality and in particular
 the notion of causality appearing in \cite{lower}. See Section~\ref{sec:Causal}.
 (Warning:  a shift of indices to  start at $(i,j,k)$ at $(1,1,1)$ is different than starting at $(0,0,0)$
 as we do here.) 
\qed
\end{remark}

 The same strategy used to prove Theorems~\ref{t:lbJ} and \ref{p:DMV-Prod+lower} also recovers well known lower bounds for
 the Schur product.

\begin{proposition}
  For psd matrices $P,Q\in \Mnc{n},$ the determinant of $S(P,Q)$ exceeds $\det(Q) \, \prod_{j=0}^{n-1} P_{jj}$
  and thus $\det(P)\, \det(Q).$
\end{proposition}

 \begin{proof}
     Factor $P=A^*A$  and $Q=B^*B$ for $A$ and $B$ upper triangular and $A_{jj},\, B_{jj} \ge 0.$
     The Schur product has rank one Choi-Kraus $S(A,B) =V^* (A\otimes B)V$ with $V=\sum_{j=0}^{n-1} (e_j\otimes e_j)e_j^*$ as described in Subsection~\ref{sss:Schur-prod}.
     A computation, using $A$ and $B$ are upper triangular gives,
    \[
    w_j:= (A\otimes B) Ve_j = (A\otimes B) e_j\otimes e_j = Ae_j\otimes Be_j = \sum_{u\le j} \sum_{v\le  j} A_{uj}B_{v j} e_u \otimes e_v.
    \]

Let $\alpha_j = (Ae_j) \|Ae_j\|^{-1}.$ Thus $\alpha_j$ is a unit vector in the 
direction of the $j$-th column of $A$ so that $\sum_{u\le j} \overline{\alpha_{j,u}^*} A_{uj} =\|Ae_j\|.$  For instance, $\alpha_0=e_0$ and 
$\alpha_1= \alpha_{10}e_0+\alpha_{11}e_1.$
Set 
\[
 R = \sum  (e_r \otimes \alpha_r) e_r^* 
\]
 and verify,
\[
 T:=R^* (A\otimes B) V =\sum_r   e_r (\alpha_r^* \otimes e_r^*)   \sum_j w_j e_j^* =  
    \sum_j \sum_{u,r\le j}   (\alpha_r^* e_u)  A_{uj} B_{rj} e_r e_j^* 
\]
 is upper triangular with diagonal entries,
\[
 T_{jj} =  (\sum_{u\le j} \overline{\alpha_{ju}}A_{uj})  B_{jj}  =\|Ae_j\| B_{jj}.
\]
Hence, since $S(P,Q)\succeq T^*T,$ 
\[
\det(P,Q)\succeq \det(T^*T)  =|\det(T)|^2 = \prod \|Ae_j\|^2 \, \prod  |B_{jj}|^2 
 = \prod P_{jj}\, \det(Q). \qedhere
\]
 \end{proof}

\section{Causal products}
\label{sec:Causal}

The notion of a causal product appearing in \cite{lower}, which we call here \df{permutation causal},
has several natural extensions to 
the category of \jcp products. Intuitively, a  
 product $\Phi:\Mnc{n}\times \Mnc{n}\to \Mnc{n}$ is causal if the  $(i,j)$ entry of $\Phi(A,B)$
 depends only on entries of $A$ and $B$ prior to $(i,j)$-th entry. 

 Formally we define left and right causal products (see Subsection~\ref{sec:causal}) in terms of left and right {\it upper triangular}
 matrices in $\Mnmc{n^2}{n}$ (see Subsection~\ref{sec:left-right}) and a product is causal if it is both left and right causal.  The 
 connection between causal and permutation causal products is given in Proposition~\ref{p:perm-caus}. Throughout the section, for $\ell=0,\dots,n-1$ we define $P_\ell := (I_{\ell+1} \oplus 0_{n-\ell-1})$\index{$P_\ell$}. 

\subsection{Left and right upper triangular matrices}
\label{sec:left-right}

A matrix $V\in \CC^{n^2 \times n}$ is 
\df{left upper triangular} (with respect to the orthonormal basis
$\{e_0,\dots,e_{n-1}\}$  of $\CC^n$) if, 
for each $k =0,\dots,n-1$ and all $i>k,$ 
\begin{equation}
    \label{e:V-upper}
\langle Ve_k, e_i \otimes e_j \rangle = 0. 
\end{equation}
For example, with $\vn=3$ and using the ordering convention for tensor products
from Subsection~\ref{sec:guide}, the transpose of a left upper triangular $\vn^2 \times \vn$ matrix $V$ has the form
\[
V^\top = \begin{pmatrix}
    a_{0,0} & a_{0,1} & a_{0,2} & 0 & 0 & 0 & 0 & 0 & 0 \\
    b_{0,0} & b_{0,1} & b_{0,2} & b_{1,0} & b_{1,1} & b_{1,2} & 0 & 0 & 0\\
    c_{0,0} & c_{0,1} & c_{0,2} & c_{1,0} & c_{1,1} & c_{1,2} & c_{2,0} & c_{2,1} & c_{2,2}\\
\end{pmatrix}.
\]

A matrix $V\in \CC^{\vn^2\times \vn}$  is \df{right upper triangular} if 
for each $k =0,\dots,n-1$ and all $j>k,$
\begin{equation}
    \label{e:V-lower}
\langle Ve_k, e_i \otimes e_j \rangle = 0. 
\end{equation}
In the case that $n=3$, the transpose of a right upper triangular matrix $V$ has the form
\[
V^\top = \begin{pmatrix}
    a_{0,0} & 0 & 0 & a_{1,0} & 0 & 0 & a_{2,0} & 0 & 0 \\
    b_{0,0} & b_{0,1} & 0 & b_{1,0} & b_{1,1} & 0 & b_{2,0} & b_{2,1} & 0\\
    c_{0,0} & c_{0,1} & c_{0,2} & c_{1,0} & c_{1,1} & c_{1,2} & c_{2,0} & c_{2,1} & c_{2,2}\\
\end{pmatrix}.
\]

 Finally, the matrix $V$ is \df{upper triangular} if it is both
 left and right upper triangular. For example, in the $n=3$ case, $V$ 
 is upper triangular means its transpose has the form 
 \[
V^\top = \begin{pmatrix}
    a_{0,0} & 0 & 0& 0 & 0 & 0 & 0 & 0 & 0 \\
    b_{0,0} & b_{0,1} & 0 & b_{1,0} & b_{1,1} & 0& 0 & 0 & 0\\
    c_{0,0} & c_{0,1} & c_{0,2} & c_{1,0} & c_{1,1} & c_{1,2} & c_{2,0} & c_{2,1} & c_{2,2}\\
\end{pmatrix}.
\]
 
Letting $\Pi_\vn : \mathbb{C}^n \otimes \mathbb{C}^n \mapsto \mathbb{C}^n \otimes \mathbb{C}^n$
\index{$\Pi_n$} denote the \df{canonical shuffle} defined by $\Pi_\vn (e_i \otimes e_j) = e_j \otimes e_i$, a matrix $V$ is left upper triangular if and only if $\Pi_\vn V$ is right upper triangular. 

\begin{lemma}
 \label{lemma:UpperTriAltDef}
 A matrix  $V\in \Mnmc{n^2}{n}$
is left upper triangular if and only if for each pair $k,\ell \in \{0,1,\dots, \vn-1\}$ with $k \leq \ell,$ 
    \[
    (P_\ell \otimes I_\vn)Ve_k = Ve_k.
    \]
    Equivalently,
    \[
    ((I-P_\ell) \otimes I_\vn)Ve_k = 0.
    \]
    In particular,
     \[
    (P_\ell \otimes I_\vn)VP_\ell= V P_\ell.
    \]

The matrix $V$ is upper triangular if and only if 
\[
V e_k \in \mathrm{span} (\{e_i \otimes e_j\}_{i,j \leq k}),
\]
for each  $k=0,1,\dots, \vn-1.$
\end{lemma}

\begin{proof}
    First suppose that $(P_\ell \otimes I_\vn)Ve_k = Ve_k$ for all $k$ with $\ell \geq k$. For  $\ell = k$ and $i > k,$  
    \[
    \langle V e_k, e_i \otimes e_j \rangle = \langle (P_k \otimes I_\vn)Ve_k, e_i \otimes e_j \rangle = \langle Ve_k, P_k e_i \otimes e_j \rangle = 0,
    \]
    where the last equality follows from $P_k e_i =0$ whenever $i > k$. 

    Now suppose $V$ is left upper triangular. Since  $(I-P_\ell) e_i = e_i$ for $i > \ell,$ if $i>\ell \geq k,$ then 
    \[
    0 = \langle Ve_k, e_i \otimes e_j \rangle = \langle Ve_k, (I-P_\ell) e_i \otimes e_j\rangle = \langle  ((I-P_\ell) \otimes I_\vn)Ve_k, e_i \otimes e_j \rangle. 
    \]
    On the other hand, if $i \leq \ell$, then $(I-P_\ell) e_i =0$ from which it follows that
    \[
    0 = \langle  ((I-P_\ell)\otimes I_\vn) Ve_k, e_i \otimes e_j \rangle.
    \]
    Thus $0 = \langle ((I-P_\ell) \otimes I_\vn)Ve_k, e_i \otimes e_j \rangle$ for all $i,j,$ which implies $((I-P_\ell) \otimes I_\vn)Ve_k = 0$. 

    That $V$ is upper triangular if and only if $V e_k \in \mathrm{span} (\{e_i \otimes e_j\}_{i,j \leq k})$ for all $k$ is immediate
    from equations~\eqref{e:V-upper} and \eqref{e:V-lower}.
\end{proof}

\subsection{Left and right causal products}
\label{sec:causal}
A product $\Phi:\Mnc{\vn}\times \Mnc{\vn}\to \Mnc{\vn}$ is \df{left causal} if it is a \jcp product and has a Choi-Kraus
representation as in item~\ref{it:CPPKraus} of Theorem~\ref{thm:CPPForms} where
each $V_k$ is left upper triangular.  The notions of \df{right causal}
and \df{causal} are defined similarly.   The following theorem
says that $\Phi$ is left causal if and only if $\Phi(A,B)_{j,\ell}$
depends only on the entries $A_{i,k}$ for $i\le j$ and
$k\le \ell;$  and that 
if $\Phi$ is left causal, then in every Choi-Kraus representation
of $\Phi$ the coefficients $V_k$ are left upper triangular.

\begin{theorem}
 \label{t:causal}
    Let $\Phi:\Mnc{\vn}\times \Mnc{\vn}\to \Mnc{\vn}$ be a \jcp product with Choi-Kraus form 
\[
 \Phi(A,B) = \sum_k V_k^* (A \otimes B) V_k,
\]
where $V_k \in \mathbb{C}^{\vn^2 \times \vn}$ for each $k.$ Also define $P_\ell = (I_{\ell+1} \oplus 0_{n-\ell-1})$ for each $\ell = 0,\dots,n-1.$ 
The following are equivalent.

\begin{enumerate}[(i)]\itemsep=6pt
\item  \label{i:causal:i}
  The product $\Phi$ is left causal;
\item \label{i:causal:ii}
 Each  $V_k$ is left upper triangular;  and 
\item \label{i:causal:iii} For each $j,\ell \in {0,\dots, \vn-1},$ 
\[
P_j \Phi(A,B) P_\ell =P_j \Phi ( P_j A P_\ell, \, B ) P_\ell. 
\]
\end{enumerate}
\end{theorem} 

\begin{proof}
Item~\ref{i:causal:ii} immediately implies item~\ref{i:causal:i}.
To prove that item~\ref{i:causal:i} implies item~\ref{i:causal:iii},
 suppose that  $\Phi$ is left causal. Hence $\Phi$ has a Choi-Kraus representation
 \[
  \Phi(A,B) =\sum_k W_k^* (A\otimes B) W_k,
 \]
  where each $W_k$ is left upper triangular. 
  By Lemma~\ref{lemma:UpperTriAltDef}  $(P_\ell \otimes I) W_k e_j = W_k e_j$ for all
   $0\le \ell \le \vn-1$ and $0\le j \leq \ell$. Hence,
 \[
    (P_\ell \otimes I) W_k P_\ell = W_k P_\ell,
 \]
 for all $0\le \ell \le \vn-1.$ 
 It follows, for all $0\le j,\ell \le \vn-1,$ that 
 \begin{align*}
  P_j\Phi(A,B) P_\ell   &= \sum_k P_j W_k^* (A\otimes B) W_k P_\ell\\
  &  = \sum_k P_j W_k^* (P_j \otimes I_\vn) (A \otimes B) (P_\ell \otimes I_\vn) W_k P_\ell
  \\&=  \sum_k P_j W_k^* \big(\big(P_j A P_\ell\big) \otimes B\big) W_k P_\ell
  \\& =P_j \Phi\big ( P_j A P_\ell, \, B \big )P_\ell.
 \end{align*}

  To prove item~\ref{i:causal:iii} implies item~\ref{i:causal:ii},
  suppose that $V_{k_0}$ is not left upper triangular for some $k_0$. That is, there exist indices $i < r$ and $j$ such that
 \[
    \langle V_{k_0} e_i, e_r \otimes e_j \rangle = c \neq 0. 
 \]
 Taking $A = e_r e_r^*$ and $B = e_j e_j^*$ we have that
 \begin{align*}
    \langle V_{k_0}^* (A \otimes B) V_{k_0} e_i, e_i \rangle &= \langle (e_r e_r^*) \otimes (e_j e_j^*) V_{k_0} e_i, V_{k_0} e_i \rangle \\
    & = \langle (e_r \otimes e_j) (e_r \otimes e_j)^* V_{k_0} e_i, V_{k_0} e_i \rangle \\
    & = \langle (e_r \otimes e_j)^* V_{k_0} e_i, (e_r \otimes e_j)^* V_{k_0} e_i \rangle \\
    & = |c|^2 > 0. 
 \end{align*}
 Noting that $P_i e_i = e_i$, it follows that
 \[
\Big\langle \sum_k P_i V_{k}^* (A \otimes B) V_{k} P_i e_i, e_i\Big\rangle > 0.
 \]
 Hence 
 \[
P_i  \Phi(A, B) P_i = P_i \left(\sum_k V_{k}^* (A \otimes B) V_{k}\right)P_i\neq 0.
 \]
 However, since $r > i$, we obtain $P_i e_r e_r^* P_i = 0$ from which it follows that
 \[
P_i\Phi (P_i A P_i, B )P_i = P_i\Phi(0,B)P_i = 0. 
 \]
We conclude
\[
P_i  \Phi(A, B) P_i  \neq P_i\Phi (P_i A P_i, B )P_i = P_i\Phi(0,B)P_i.
\]
That is, item~\ref{i:causal:iii} does not hold. 
\end{proof}

\begin{corollary}
 \label{c:causal}
    Let $\Phi, V_k,$ and $P_\ell$ be as in the statement of Theorem~\ref{t:causal}. Then
the following are equivalent.

\begin{enumerate}[(i)]\itemsep=6pt
\item  
  The product $\Phi$ is right causal;
\item 
 Each  $V_k$ is right upper triangular;  and 
\item For each $j,\ell \in {0,\dots, \vn-1},$ 
\[
P_j \Phi(A,B) P_\ell =P_j \Phi\big ( A, P_jB P_\ell ) P_\ell. 
\]
\end{enumerate}
\end{corollary}
\begin{proof}
Combine Theorem~\ref{t:causal} with the facts that 
$\Phi$ is right causal if and only if the product $\Phi_\Pi$ defined by 
 $\Phi_\Pi(A,B)=\Phi(B,A)$  is left causal; and 
 $\Phi(A,B) = \sum_k V_k^* (A \otimes B) V_k$ is a Choi-Kraus representation for $\Phi$ if and only if
 \[
\sum_k V_k^* \Pi_\vn (A \otimes B) \Pi_\vn V_k
 \]
 is a Choi-Kraus representation for $\Phi_\Pi$.  
 \end{proof}

\begin{remark}\rm
\label{r:more-general}
  The notion of a causal product generalizes as follows.  Given families $\mathcal{L}$ and $\mathcal{R}$ of
  matrices, let 
\[
  \mathcal{C}_{\mathcal{L},\mathcal{R}} =\{\Phi:   L\, \Phi(A,B) \, R =  L\Phi(LAR,B)R\}.
\]
  Given a collection $\mathcal{F}$ of \jcp products $\Phi,$ there are maximal pairs $(\mathcal{L},\mathcal{R})$ such
  that
\[
 \mathcal{F} \subseteq \mathcal{C}_{\mathcal{L},\mathcal{R}}.
\]
 There are a number of natural related questions that we do not pursue here; e.g., when does a maximum
 pair $(\mathcal{L},\mathcal{R})$ exist?  
 \qed
\end{remark}

 \subsection{Permutation causal products}
We pause here to   caution that there is a stronger (more restrictive notion) of causal product that appears in 
 \cite{dominiqueOppenheim} that  we refer to here as
 \df{permutation  causal} and that we  now describe in the notation of \cite{dominiqueOppenheim}. 
For a given integer $\vn \geq 0$ and for every $\ell = 0,\dots \vn-1$ fix 
\begin{enumerate}
    \item a set $T_\ell \subset \{0,...,\ell\}$ with $\ell \in T_\ell$
    \item a permutation $\sigma_\ell: T_\ell \to T_\ell$. 
\end{enumerate}
The product $\star:\Mnc{vn}\to \Mnc{vn}$ defined  by 
\[
(A \star B)_{jk} := \sum_{(p,q) \in T_j \times T_k} a_{p,q} b_{\sigma_j(p),\sigma_k(q)}
\] 
for $A, B \in \mathbb{C}^{\vn \times \vn}$ is a \df{permutation causal} product.

\begin{proposition}
\label{p:perm-caus}
    A permutation causal product $\star$ is \jcp and causal.  Moreover, the Choi-Kraus rank of 
    $\star$ is one.
\end{proposition}

\begin{proof}
Given a permutation causal product $\star,$ let $V \in \mathbb{C}^{\vn^2 \times \vn}$ be the matrix with $\ell$-th column
\[
 Ve_\ell = \sum_{j \in T_\ell} e_j \otimes e_{\sigma_\ell(j)}.
\]
It is straightforward to check that $V$ and $\Pi_\vn V$ are both left
upper triangular and hence $V$ is upper triangular. Thus, to complete
the proof it suffices to show that 
\[
 A \star B = V^* (A\otimes B)V
\]
for all $A,B \in \mathbb{C}^{\vn \times \vn}$. To verify this
claim, observe
\begin{align*} 
(V^* (A \otimes B)V)_{j k} &= \big\langle \big(A \otimes B\big) \big(\sum_{p \in T_j} e_p \otimes e_{\sigma_j (p)}\big),\big(\sum_{q \in T_k} e_q \otimes e_{\sigma_k (q)}\big)\big\rangle \\
& = \sum_{p \in T_j} \sum_{q \in T_k} \langle A e_p \otimes B e_{\sigma_j (p)}, e_q \otimes e_{\sigma_k(q)}\rangle \\
&= \sum_{(p,q) \in T_j \times T_k} \langle Ae_p, e_q \rangle \langle B e_{\sigma_j(p)}, e_{\sigma_k(q)} \rangle\\
& = \sum_{(p,q) \in T_j \times T_k} a_{p,q} b_{\sigma_j(p),\sigma_k(q)} \\
&= (A \star B)_{j k}. \qedhere
\end{align*}
\end{proof}

\section{Commutativity, Associativity, and Units of \jcp products}
\label{sec:BasicAlgebra}

 In this section we study commutativity, associativity, and units of \jcp products in terms of their Choi-Kraus representations and Choi matrices.

\subsection{Commutativity}
\label{subsection:commutativity}
  Commutativity of a \jcp product $\Phi: \Mnc{n} \times \Mnc{n} \to \Mnc{m}$ has characterizations
  in terms of the canonical shuffle and both its Choi matrix  and the existence of a suitable
  Choi-Kraus form. See Proposition~\ref{prop:CommutativeChoiMat} and Theorem~\ref{thm:CommutativeChoiRep}, respectively.

 As usual, $\Pi_n : \mathbb{C}^n \otimes \mathbb{C}^n \mapsto \mathbb{C}^n\otimes \mathbb{C}^n$ denotes the canonical shuffle $\Pi_n(A\otimes B)\Pi_n= B\otimes A$. 

\begin{proposition}
\label{prop:CommutativeChoiMat}
     If  $\Phi: \Mnc{n} \times \Mnc{n} \to \Mnc{m}$ is  a \jcp product with Choi matrix $C_\Phi,$ then $\Phi$ is commutative if and only if $(\Pi_n \otimes I_m) C_\Phi (\Pi_n \otimes I_m) = C_\Phi.$ 
\end{proposition}

\begin{proof}
    It is straightforward to check that $(\Pi_n \otimes I_m) C_\Phi (\Pi_n \otimes I_m) = C_\Phi$ if and only if $\Phi(E_{ij},E_{s,t}) = \Phi(E_{s,t}, E_{i,j})$ for all $i,j,s,t$. From the bilinearity of $\Phi$, the latter condition holds if and only if $\Phi$ is commutative. 
\end{proof}

Because of choices involved in passing from a psd Choi matrix to
a Choi-Kraus representation, describing commutative products in terms
of Choi-Kraus coefficients is more subtle. 

\begin{theorem}
\label{thm:CommutativeChoiRep}
     A \jcp product $\Phi: \Mnc{n} \times \Mnc{n} \to \Mnc{m}$ is commutative if and only if  it possesses  a Choi-Kraus representation $\Phi(A,B) =\sum_{j} V_j^* (A \otimes B) V_j$ that satisfies $\Pi_n V_j = \pm V_j$ for each $j$.  If $C_\Phi$ has rank one and $\Phi$ is commutative, then any minimal Choi-Kraus representation $\Phi(A,B) = V^* (A \otimes B) V$ for $\Phi$ satisfies $\Pi_n V = \pm V$.
\end{theorem}

\begin{remark}
    It is straightforward to give an example of a Choi-Kraus representation of commutative product $\Phi(A,B) = \sum_{j} V_j^* (A \otimes B) V_j$ that does not satisfy $\Pi V_j = \pm V_j$ for each $j$. A simple example  where $\Phi:\Mnc{2}\times \Mnc{2}\to \CC$   is obtained by taking
    \[
    V_1^\top = \begin{pmatrix}
        0 & 1 & 0 & 0 \\
    \end{pmatrix} \qquad \mathrm{and} \qquad     V_2^\top = \begin{pmatrix}
        0 & 0 & 1 & 0 \\
    \end{pmatrix}.
    \]
    However, one has $\Pi_2 V_1 = V_2$ and $\Pi_2 V_2 = V_1$. 
    
     Given a Choi-Kraus representation $\Phi(A,B) = \sum V_j^* (A \otimes B) V_j$ of a commutative \jcp product, one may wonder if there is always a permutation $\sigma$ such that $\Pi_n V_j = \pm V_{\sigma(j)}$. The
      answer is no. Intuitively, if one expresses $C_\Phi = WW^*$ where the columns $w_j$ of $W$ are eigenvectors of $C_\Phi$, then the invariance of the eigenvectors under left multiplication by $(\Pi_n \otimes I_m)$ leads to a Choi-Kraus representation with the desired form. However, letting $U$ be an arbitrary unitary, we also have $C_\Phi = (WU) (WU)^*$, which leads to an alternative Choi-Kraus representation.  
      The columns of $WU$ need no longer be invariant under left multiplication by $(\Pi_n \otimes I_m).$ Hence generically the 
      coefficients in the resulting Choi-Kraus representation need not 
      respect the canonical shuffle $\Pi_n$. Based on
     this observation is not hard to construct a numerical example
     of a Choi-Kraus representation for a commutative product that fails
     to be invariant under $\Pi_n.$ An example is provided in a public Mathematica notebook, which can be found at \href{https://github.com/NCAlgebra/UserNCNotebooks/tree/master/EvertJuryMcCullough}{https://github.com/NCAlgebra/UserNCNotebooks}.
    \qed
\end{remark}

    The proof of Theorem~\ref{thm:CommutativeChoiRep} relies on two lemmas.

\begin{lemma}
    \label{lem:preCommutativeChoiRep}
       If  $C_\Phi$ has distinct (and) nonzero eigenvalues and $\Phi$ is commutative, then the conclusion of Theorem~\ref{thm:CommutativeChoiRep} holds.
\end{lemma}

\begin{proof}
     Suppose that $C_\Phi$ has distinct nonzero eigenvalues. Following the proof of Theorem \ref{thm:CPPForms}, a Choi-Kraus representation for $\Phi$ can be obtained by computing an eigendecomposition 
      of $C_\Phi$. That is, if  $(\lambda_j,w_j)$ is the $j$-th nonzero eigenpair of $C_\Phi$  with $w_j$
      a unit vector and  $V_j$ is
      the $n^2$ by $m$ matrix so that $\sqrt{\lambda_j} w_j = \vec(V_j^*)$, then
    \[
    \Phi(A,B) =\sum_{j} V_j^* (A \otimes B) V_j.
    \]
    Furthermore, using $(\Pi_n \otimes I_m) C_\Phi (\Pi_n \otimes I_m) = C_\Phi$, it is straightforward to check that $(\lambda,w)$ is an eigenpair of $C_\Phi$ if and only if $(\lambda, (\Pi_n \otimes I_m) w)$ is an eigenpair of $C_\Phi$. Since $C_\Phi$ is assumed to have distinct eigenvalues, it follows that for each $j$, there exists a $\theta_j \in [0,2\pi)$ such that
    \[
    (\Pi_n \otimes I_m) w_j = e^{i \theta_j} w_j.
    \]
    Left multiplying this equality by $\Pi_n \otimes I_m$ now yields
    \[
    w_j = (\Pi_n \otimes I_m)^2 w_j = e^{i \theta_j} (\Pi_n \otimes I_m) w_j = e^{2i \theta_j} w_j.
    \]
    Thus $\theta_j \in \{0,\pi\}$ for each $j$. It follows that $(\Pi_n \otimes I_m) w_j = \pm w_j.$
Since $w_j=\sum_{i,s} e_i\otimes e_s \otimes V_j^*(e_i\otimes e_s),$
\[
 \begin{split}
   \pm \sum_{i,s} e_i\otimes e_s \otimes V_j(e_i\otimes e_s)
    & = (\Pi_n \otimes I_m)\sum_{i,s} e_i\otimes e_s \otimes V_j (e_i\otimes e_s)
    \\ & = \sum_{i,s} e_i \otimes e_s\otimes V_j^*(e_s\otimes e_i).
 \end{split}
\]
 Hence $V^* (e_s\otimes e_i) = \pm V^* (e_i\otimes e_s)$ for each $i,s$ from which it follows that 
\[
 \Pi_n V_j = \Pi_n \sum (e_i\otimes e_s) (e_i\otimes e_s)^* V_j = \pm \sum (e_i\otimes e_s) (e_i\otimes e_s)^* V_j =\pm V_j.
\]

 Finally, if $C_\Phi$ has rank one and $V^* (A \otimes B) V$ is a Choi-Kraus representation of $\Phi$, then $w = \vec(V^*)$ must be an eigenvector of $C_\Phi$ corresponding to the nonzero eigenvalue, hence this case follows from the previous argument. 
\end{proof}

    \begin{lemma}
    \label{lem:CommChoiDistinctEig}
    For any $n,m \in \mathbb{N}$, there exists a commutative \jcp product $\Psi: \Mnc{n} \times \Mnc{n} \mapsto \Mnc{m}$ such that the Choi matrix $C_\Psi$ has full rank and such that all eigenvalues of $C_\Psi$ are distinct (and non-zero). 
    \end{lemma}

    \begin{proof}
        By Proposition~\ref{prop:CommutativeChoiMat}, it is sufficient to produce a positive definite  matrix $C$, with distinct and non-zero eigenvalues, such that $(\Pi_n \otimes I_m) C (\Pi_n \otimes I_m) = C$. To this end consider the orthonormal basis
    \[
      F= \bigcup_{k=0}^{m-1} F_k,
    \]
    where 
        \[
        F_k = \{(e_i \otimes e_j \otimes e_k+e_j \otimes e_i \otimes e_k)/\sqrt{2}\}_{i<j} \cup \{(e_i \otimes e_j \otimes e_k-e_j \otimes e_i\otimes e_k)/{\sqrt{2}} \}_{i<j} \cup \{e_i \otimes e_i \otimes e_k\}_i.
        \]
        Let $F = \{f_\ell\}_{\ell=1}^{mn^2}$ denote a labelling of the basis elements, and let $\{\alpha_\ell\}_{\ell=1}^{mn^2} \subset \mathbb{N}$ be a list of distinct positive numbers. Thus the matrix
        \[
        C:= \sum_{\ell=1}^{mn^2}  \alpha_\ell f_\ell f_\ell^*
        \]
        is PD and has full rank $mn^2$. Furthermore, the $\alpha_\ell$ are precisely the eigenvalues of $C$, hence $C$ has distinct eigenvalues. Letting $\Psi$ denote the \jcp product with Choi matrix $C$ produces a \jcp product with the desired properties. 
\end{proof}

\begin{proof}[Proof of Theorem~\ref{thm:CommutativeChoiRep}]
    Let $\Psi$ be the product constructed in Lemma \ref{lem:CommChoiDistinctEig}. For $\epsilon \geq 0$, Define a matrix product $\Phi_\epsilon:\Mnc{n} \times \Mnc{n} \mapsto \Mnc{m}$ by $\Phi_\epsilon(A,B) = \Phi(A,B)+\epsilon \Psi(A,B)$. Then for all but finitely many $\epsilon$, the Choi Matrix $C_{\Phi_\epsilon} = C_{\Phi} +\epsilon C_\Psi$ has distinct eigenvalues, all of which are positive. Letting $\epsilon$ go to zero and taking limits of (subsequences) of the Choi-Kraus representations coming from Lemma~\ref{lem:preCommutativeChoiRep} yields a Choi-Kraus rep for $\Phi$ that has the desired form. 

    The case that $C_\Phi$ has rank one follows directly from Lemma~\ref{lem:preCommutativeChoiRep}. No approximation is required.
\end{proof}

\subsection{Associativity}

We now consider associativity for \jcp matrix products. As with commutativity, the situation is more straightforward for Choi-Kraus rank one \jcp products. 

\begin{proposition}
    If   $\Phi:\Mnc{n} \times \Mnc{n} \to \Mnc{n}$ is  a \jcp product with Choi-Kraus rank one representation
    \[
    \Phi(A,B) = V^* (A \otimes B) V,
    \]
    then $\Phi$ is associative if and only if 
    \[
    (I_n \otimes V)V = \omega (V \otimes I_n) V,
    \]
    for some unimodular constant $\omega.$
\end{proposition}

\begin{proof}
  Suppose $\Phi(A,B)=V^*(A \otimes B)V$ is associative. Comparing
    \[
    \Phi(\Phi(A,B),C) = V^*(V^* \otimes I_n) (A \otimes B \otimes C) (V \otimes I_n) V,
    \] 
    and
    \[
    \Phi(A,\Phi(B,C)) = V^*(I_n \otimes V^*) (A \otimes B \otimes C) (I_n \otimes V) V,
    \]
    obtains
    \[
    V^*(V^* \otimes I_n) (A \otimes B \otimes C) (V \otimes I_n) V=V^*(I_n \otimes V^*) (A \otimes B \otimes C) (I_n \otimes V)V
    \]
    for all matrices $A,B,C \in \Mnc{n}$. 
    In particular, the  result is immediate if either $(I_n\otimes V)V$ or $(V\otimes I_n)V.$ Otherwise, 
    since tensor products of matrix units span $\Mnc{n^3}$, these Choi-Kraus forms induce the same Choi-Kraus rank one cp map $\Psi:\Mnc{n^3} \to \Mnc{n}$. Furthermore, since $\Psi$ has Choi-Kraus rank one, any two Choi-Kraus representations of $\Psi$ must be equal up to a phase (a unimodular complex number) $\omega,$ from which the forward direction follows. It is straightforward to check the reverse direction. 
\end{proof}

\begin{remark}
    Given any unimodular constant $\omega$, there exists an associative \jcp product $\Phi:M_n \times M_n \to M_n$ with Choi-Kraus representation $\Phi(A,B) = V^*(A \otimes B) V$ such that $(I_n \otimes V)V = \omega(V \otimes I_n)V$. A concrete example is obtained by letting $n=3$ and defining a linear map $V \in \mathbb{C}^{n^2 \times n}$  by $Ve_1 =0$ and $V e_2 = e_1 \otimes e_1$ and $V e_3 = e_1 \otimes e_2 + \overline{\omega} e_2 \otimes e_1$. This example was generated with the aid of ChatGPT and was verified in Mathematica. A public Mathematica notebook that verifies the example is available at \href{https://github.com/NCAlgebra/UserNCNotebooks/tree/master/EvertJuryMcCullough}{https://github.com/NCAlgebra/UserNCNotebooks}.
\end{remark}

\subsection{Units}

A \df{left unit} (resp. \df{right unit}) for a product $\Phi:\Mnc{n}\times \Mnc{n}\to \Mnc{n}$ is
a matrix $\unitL\in \Mnc{n}$  (resp. $\unitR$) 
satisfying $\Phi(\unitL,Y)=Y$ (resp. $\Phi(X,\unitR)=X$) for all $Y$ (resp. $X)$ in $\Mnc{n}.$  A \df{unit}
$\unitT$ is a matrix that is both a left and right unit.   Similarly a matrix $N$ is \df{left null} for $\Phi$
if $\Phi(N,Y)=0$ for all $Y.$ The notions of a \df{right null} and \df{null} matrix are defined similarly.

\begin{lemma}
\label{l:units}
  For a \jcp product $\Phi:\Mnc{n}\times \Mnc{n}\to \Mnc{n}$  the set of left (resp. right) null matrices is a self-adjoint subspace of $\Mnc{n}.$

  If $\Phi$ has a left (resp. right) unit, then $\Phi$ has no non-trivial right (resp. left) null matrices. In
  particular, if $\Phi$ has a unit, then $\Phi$ has no left or right null vectors and this unit is unique as a unit and as a left and right unit.
\end{lemma}

\begin{proof}
    The set $\mathcal{N}$ of left null matrices is evidently a subspace of $\Mnc{n}.$ If $N\in \mathcal{N}$ and $B\in \Mnc{n},$
    then, by considering a Choi-Kraus representation for $\Phi,$
\[
 \Phi(N^*,B^*)^* = \Phi(N,B)=0.
\]
Hence $N^*\in\mathcal{N},$ and the first statement is proved. The second statement is evident.
\end{proof}

\begin{example}
 \label{eg:not-unique-psd-units}
  The product $\Phi:\Mnc{3}\times \Mnc{3}\to \Mnc{3}$ defined by $\Phi(X,Y)= (X_{0,0} + X_{1,1}) Y$ is \jcp with Choi-Kraus rank two determined by the Choi-Kraus operators $V_k:\CC^3\to \CC^3\otimes\CC^3$ defined by $V_k e_j=e_k \otimes e_j$  for $0 \leq k \leq 1$ and $0 \leq j \leq 2.$ Alternately,  $V_k = e_k \otimes I_3$   for $k=0,1.$ 
  The matrix 
\[
   \unitL = \begin{pmatrix} \frac12 &0&0 \\ 0 &\frac12 &0 \\ 0 &0&0\end{pmatrix}
\]
  is a left unit.   The subspace of left null matrices for $\Phi$ is the span $E_{0,0} -E_{1,1}$ together with the matrix units
  $E_{i,j}$  excluding $E_{0,0}$ and $E_{1,1}$.
In particular, $E_{0,0}$   and
\[
 \unitL^\prime = \frac12 \begin{pmatrix} 1& 1& 0 \\ 1&1&0\\0&0&0\end{pmatrix}
\]
 are rank one (psd) and also a left \matus.   
 \qed
\end{example}

\begin{example}
\label{eg:variant:not-unique}
    Consider the following variant $\Phi:\Mnc{2}\times \Mnc{2}\to \Mnc{2}$ of the Schur product map $S:\Mnc{2}\times \Mnc{2}\to \Mnc{2}$,
\[
 \Phi(A,B) = \begin{pmatrix} 2 A_{0,0}B_{0,0} & A_{0,1}B_{0,1} \\ A_{1,0} B_{1,0}  & 2 A_{1,1} B_{1,1} \end{pmatrix}.
\]
Thus 
\[
  \Phi(A,B) = S(A,B) +  \begin{pmatrix} A_{0,0} B_{0,0} &0\\0& A_{1,1}B_{1,1} \end{pmatrix}.
\]
   The Choi matrix for
$\Phi$ is
\[
 C_\Phi = C_S + E_{0,0}\otimes  \left ( E_{0,0}\otimes E_{0,0} \right ) \ +  \  E_{1,1}\otimes \left (E_{1,1}\otimes E_{1,1}\right ),
\]
 where $C_S$ is the Choi matrix for the Schur product $S.$ 
 Since this matrix is psd, $\Phi$ is a \jcp product.   Observe that
 \[
  \unitT =  \frac12 \begin{pmatrix} 1 & 2 \\ 2 & 1 \end{pmatrix} 
 \]
  is a unit (both left and right) for $\Phi.$  In particular, by Lemma~\ref{l:units},
  $\Phi$ is unique both as a left and as right unit.

   More generally, if  $P$ a psd matrix with non-zero entries,  then $\Phi:\Mnc{n}\times \Mnc{n}\to \Mnc{n}$ defined by
   $S(P,S(A,B))$ is a \jcp product with (two sided, unique) unit 
   \[
   1./P = \begin{pmatrix} 1/P_{i,j} \end{pmatrix}_{i,j=0}^{n-1}.
   \]
    The matrix $1./P$ is psd if and only if $P$ is rank one.
  \qed
\end{example}

\begin{lemma}
\label{l:LU}
  Suppose $\Phi:\Mnc{n}\times \Mnc{n} \to \Mnc{n}$  is \jcp.  
\begin{enumerate}[(i)]\itemsep=5pt
\item   \label{i:LU:i}
   If $\Phi$ has a left unit,  then it has a self-adjoint left unit.

\item \label{i:LU:iii}
  If $\Phi$ has a psd left unit, then $\Phi$ has a rank one psd left unit.  In fact, 
if $\Phi$ has a psd left unit of rank at least $2$, then $\Phi$ has infinitely many distinct
  rank one psd \matus.
\end{enumerate}
\end{lemma}

The proof of Lemma~\ref{l:LU} will take advantage of the following routine lemma.

\begin{lemma}
\label{l:pre:LU}
    If  $\sum W_k^* A W_k$ is a Choi-Kraus representation for the identity $\iota:\Mnc{n}\to \Mnc{n},$ then 
    each $W_k$ is a multiple of the identity.
\end{lemma}

\begin{proof}
    For each vector  $y\in \CC^n,$
\[
 yy^*    =\iota(yy^*) = \sum_k W_k^* yy^* W_k  = \sum_k z_k z_k^*,
\]
 where $z_k= W_k^* y.$ Hence for each $k$ each $y$ is an eigenvector of $W_k^*.$ The result follows.
\end{proof}

\begin{proof}[Proof of Lemma~~\ref{l:LU}]
 From a Choi-Kraus form for $\Phi$ (or just positivity), $\Phi(\unitL^*,Y)^* = \Phi(\unitL,Y^*)=Y^*.$
 Thus $\Phi(\unitL^*,Y)=Y$ and hence $\unitL^*$ is also a left unit. It follows that  $\frac12 (\unitL + \unitL^*)$ is a left unit which proves item~\ref{i:LU:i}.
 
 To prove item~\ref{i:LU:iii}, suppose $\Phi$ has a 
 psd left \matu $\unitL$ of rank $r \geq 2.$  Thus
 there exist orthogonal vectors $v_1,\dots,v_r\in \CC^n$ 
 such that $\unitL = \sum v_j v_j^*.$  There exist finitely many  $V_\ell:\CC^n\to \CC^n\otimes \CC^n$ such that
\[
 \Phi(X,Y) = \sum_\ell  V_\ell^* (X\otimes Y) V_\ell.
\]
Since 
\[
  Y= \iota(Y):= \Phi(\unitL,Y) =  \sum_{k,\ell} V_\ell^* (v_k\otimes I_n) Y  (v_k\otimes I_n)^* V_\ell
   =\sum_{k,\ell} W_{k,\ell}^* Y W_{k,\ell},
\]
 where $W_{k,\ell}= (v_k\otimes I_n)^* V_\ell,$ it follows from Lemma~\ref{l:pre:LU} that each $W_{k,\ell}=c_{k,\ell}I_n$
 for some scalar $c_{k,\ell}.$

 Choose a $k$ for which there is an $\ell$ such that $c_{k,\ell}\ne 0.$
 Let $\mu$ denote the reciprocal of the square root of  $\sum_{\ell} |c_{k,\ell}|^2.$ 
 It now follows, with $u_k= \mu v_k,$ that
\[
  \sum_{\ell} V_\ell^* (u_k u_k^* \otimes Y) V_\ell = 
   \mu^2 \, \sum_{\ell}   W_{k,\ell}^* Y W_{k,\ell} = Y. 
\]
Hence  $u_k u_k^*$ is a rank one psd left unit of $\Phi$.

Now, fix an index $j \neq k$ and let $\alpha \in \mathbb{R}$. Set $y_{\alpha} := v_k + \alpha v_j$. Observe that $(y_{\alpha}^* \otimes I_n) V_\ell = (c_{k,\ell}+\alpha c_{j,\ell}) I_n$ for all $\ell$. Furthermore, since not all $c_{k,\ell}$ are zero, there is at most one $\alpha$ for which the coefficients $c_{k,\ell}+\alpha c_{j,\ell}$ are zero for all $\ell$. We now have
\[
  \sum_{\ell} V_\ell^* (y_\alpha y_\alpha^* \otimes Y) V_\ell =  \sum_{\ell}   \overline{(c_{k,\ell}+\alpha c_{j,\ell})} Y (c_{k,\ell}+\alpha c_{j,\ell}) =  \sum_{\ell}   |c_{k,\ell}+\alpha c_{j,\ell}|^2 Y. 
\]
Thus, letting $\mu_\alpha$ be the reciprocal of $\sum_{\ell}   |c_{k,\ell}+\alpha c_{j,\ell}|^2$ for $\alpha$ such that $\sum_{\ell}   |c_{k,\ell}+\alpha c_{j,\ell}|^2 \neq 0$ gives that
\[
\{ \mu_\alpha y_\alpha y_\alpha^* : \ \sum_{\ell}   |c_{k,\ell}+\alpha c_{j,\ell}|^2 \neq 0\}
\]
is an infinite collection of distinct rank one psd units. 
\end{proof}

We now show that the existence of a left psd unit  for a \jcp product can be determined by examining a Choi-Kraus decomposition.

\begin{theorem}
    Let $\Phi(A,B) = \sum_{j=1}^\ell V_j^*(A \otimes B) V_j$ and for each $j$ write 
    \[
    V_j^* = [V_{j,0}^*\  \dots \  V_{j,n-1}^*] = \sum_{k=0}^{n-1} e_k^* \otimes V_{jk}^*,
    \]
    where $V_{jk} \in \Mnc{n}$ for all $j,k$. The map  $\Phi$ has a left psd unit if and only if there exists a vector $\alpha \in \mathbb{C}^n$  and unit vector $\lambda\in \CC^\ell,$ such that
       \begin{equation}
       \label{eq:psdunitEquation}
\sum_{k=0}^{n-1} \overline{\alpha}_k V_{j,k} = \lambda_j I_n, \quad j=1,\dots,\ell.
    \end{equation}
        In this case, $\unitL = \alpha \alpha^*$ is a rank one left psd unit. Furthermore, $\unitL$ is a unique psd unit if and only if every solution of equation \eqref{eq:psdunitEquation} has the form $\omega \alpha$ for some unimodular $\omega \in \mathbb{C}$.
\end{theorem}

\begin{proof}
    It is straightforward to verify that if $\alpha \in  \mathbb{C}^n$ is a solution of equation \eqref{eq:psdunitEquation}, then $\alpha \alpha^*$ is a rank one left psd unit of $\Phi$. It follows that if this equation has two solutions that are not unimodular multiples of each other, then $\Phi$ has multiple left psd units. 

    Now suppose that $\Phi$ has a psd unit. By Lemma \ref{l:LU}, it follows that $\Phi$ has a rank one psd unit $\unitL$. Choose $\alpha \in \mathbb{C}^n$ such that $\alpha \alpha^* = \unitL$. Since $\iota(Y)= \Phi(\unitL,Y)$ is the identity and 
\[
\begin{split}
  \iota(Y) & = \sum_j V_j^* (\alpha\alpha^*\otimes Y) V_j 
  \\ &= \sum_{j,k}  (e_k^*\otimes V_{jk}^*) (\alpha\alpha^* \otimes Y)
  (e_k\otimes V_{j,k}) = \sum_j  (\sum_k \alpha_k V_{jk}^*)  Y  (\sum_k \overline{\alpha_k} V_{jk}), 
\end{split}
\]
 Lemma~\ref{l:pre:LU} implies  each $W_j =\sum_k \overline{\alpha_k} V_{jk}$ is a multiple, say $\lambda_j,$ of the identity.
  In particular, 
    \[
    \iota(I_n) = \sum_{j=1}^\ell (\lambda_j^* I_n) (I_n) (\lambda_j I_n) = \sum_{j=1}^\ell |\lambda_j|^2 I_n = I_n,
    \]
    from which we conclude that $\sum_{j=1}^\ell |\lambda_j|^2 = 1$.
    
    To complete the proof, notice that if $\Phi$ has a psd left unit with rank two or more, then by Lemma~\ref{l:LU} item~\ref{i:LU:iii},
    $\Phi$ has at least two distinct rank one psd units. On the other hand, if $\Phi$ has more than one psd unit, then either 
    they are both rank one, or one has rank at least two.
\end{proof}

\begin{remark}
 Right units can be understood  analogously by studying the matrices $\{W_j\}_{j=1}^\ell := \{\Pi_n V_j\}_{j=1}^\ell$. In particular, $\Phi(A,B) = \sum_{j=1}^\ell V_j^*(A\otimes B)V_j$ has a right unit $\unitR$ if and only if there exists an $\alpha \in \mathbb{C}^n$ such that 
 \[
\sum_{k=0}^{n-1} \overline{\alpha}_k W_{j,k} = \lambda_j I_n \ \mathrm{for \ all\ } j=1,\dots,\ell \qquad
    \mathrm{and} \qquad  \sum_{j=1}^\ell |\lambda_j|^2 = 1.
 \]
 Here $W_j^* = [W_{j,0}^* \ \dots \ W_{j,n-1}^*]$ with $W_{j,k} \in \Mnc{n}$ for each $j,k$. 
 \qed
\end{remark}

Additionally, we show that finding left and right units is equivalent to solving a system of block matrix equations determined by the Choi matrix.

\begin{proposition}
    Let $\Phi:\Mnc{n} \times \Mnc{n} \to \Mnc{n}$ be a \jcp product and let $C_\Phi = ((\Phi(E_{i,j},E_{s,t})_{i,j=0}^{n-1}))_{s,t=0}^{n-1}$ be the Choi matrix of $\Phi$. Then $\unitL$ is a left unit of $\Phi$ if and only if $\unitL$ is a solution to the system of linear equations
    \[
    \sum_{i,j} \unitL_{i,j}\Phi(E_{i,j},E_{k,\ell}) = E_{k,\ell} \quad \mathrm{for \ all\ } k,\ell. 
    \]
    Similarly,  $\unitR$ is a right unit of $\Phi$ if and only if $\unitR$ is a solution to the system of linear equations
    \[
    \sum_{k,\ell} \unitR_{k,\ell}\Phi(E_{i,j},E_{k,\ell}) = E_{i,j} \quad \mathrm{for \ all\ } i,j. 
    \]
\end{proposition}
\begin{proof}
    Suppose $\unitL$ is a left unit of $\Phi$. Then ucp map $\Psi_\unitL:\Mnc{n} \to \Mnc{n}$ defined by $\Psi_\unitL(B) = \Phi(\unitL,B)$ is the identity map and therefore has Choi matrix $C_{\Psi_\unitL} = (E_{k,\ell})_{k,\ell}$. The proof is completed by observing that that  $\sum_{i,j} \unitL_{i,j}\Phi(E_{i,j},E_{k,\ell}) $ is precisely $k,\ell$ block of the Choi matrix $C_{\Psi_\unitL}$.
\end{proof}

\begin{question}
    What is the maximum Choi-Kraus rank of a \jcp product with a left unit? What is the maximum Choi-Kraus rank of a \jcp product with a unit?
\end{question}

\section{Separable \jcp products}
\label{sec:separable}
 Separable Choi matrices play an important role in quantum information theory. For example, they are important in the theory of entanglement breaking quantum channels; see, e.g., \cite{HSR,Watrous}. Letting $\FF$ denote $\RR$ or $\CC$, we say that a \jcp product $\Phi:\Mnf{n}\times \Mnf{n} \to \Mnf{m}$ is \df{separable over} $\mathbb{F}$ if its Choi matrix $C_\Phi$ decomposes as $C_\Phi = \sum_{j=1}^\ell w_j w_j^* \otimes v_j v_j^*$ with $w_j \in \mathbb{F}^{n^2}$ and $v_j \in \mathbb{F}^m$. If $\ell$ is as small as possible, we say that $C_\Phi$ has \df{separable rank} $\ell$ over $\mathbb{F}$.

 The following proposition interprets the usual characterization of an entanglement breaking map in terms of a Choi-Kraus
 representation in the context of \jcp products.
 
\begin{proposition}
\label{prop:SeparableKrausForm}
     A \jcp product  $\Phi: \Mnf{n} \times \Mnf{n}\to \Mnf{m}$ is separable over $\mathbb{F}$ if and only if $\Phi$ has a Choi-Kraus decomposition $\Phi(A,B) = \sum V_j^* (A \otimes B) V_j$ where each $V_j \in \mathbb{F}^{n^2 \times m}$ has rank one. Moreover, $V_j = v_j w_j^*$ for $v_j\in \FF^{n^2}$ and $w_j\in \FF^{m}$ if and only if 
      there exist $u_j\in \FF^{n^2}$ such that 
\[   
    C_\Phi = \sum u_j u_j^* \otimes w_j w_j^*.
\]
Furthermore, in this case, 
\begin{equation}
\label{eq:SeparableKrausForm} 
 \Phi(A,B)= \sum_j \left [ v_j^*(A\otimes B)v_j  \right ]  w_jw_j^*.
\end{equation}
\end{proposition}

\begin{remark}\rm
\pushQED{\qed}
  The $u_j$ can be obtained from the $v_j$ as follows. Fix an 
  orthonormal basis $\{e_1,\dots,e_n\}$ for $\FF^n.$ Decomposing the vector
  $v_j\in \FF^{n^2}=\FF^n \otimes \FF^n$  as
\[
 v_j = \sum_{a,b=0}^{n-1} v_{j,a,b}\,  e_a\otimes e_b,
\]
 the vector $u_j$ is given by 
\[
 u_j = \sum_{a,b=0}^{n-1} \overline{v_{j,a,b}} \, e_a\otimes e_b. \qedhere
\]
\popQED
\end{remark}

\begin{proof}[Proof of Proposition~\ref{prop:SeparableKrausForm}]
   Given $v\in \FF^{n^2}=\FF^n\otimes \FF^n,$ write
\[
 v = \sum_{a,b=0}^{n-1} v_{a,b} \, e_a\otimes e_b,
\]
 with respect to a fixed orthonormal basis $\{e_0,\dots,e_{n-1}\}$ of $\FF^n.$
 Given $w\in\FF^m$ and letting $V=vw^*,$
\[
  V^* (A\otimes B) V = \left [v^* (A\otimes B)v \right ] \,  ww^*
\]
for $A,B\in \Mnf{n}.$  Letting $\Phi(A,B)=V^* (A\otimes B)V$ it
follows that the Choi matrix for $\Phi$ is, by definition,  
\[
\begin{split}
 C_\Phi & = \sum (E_{ij} \otimes E_{st}) \otimes \Phi(E_{ij},E_{st})
 \\[3pt] & = \left [\sum  v^* \left (\sum E_{ij} \otimes E_{st} \right )v \,  (E_{ij} \otimes E_{st}) \right ] \otimes  ww^*
 \\[3pt] & =  \left [ \sum  v^*(e_i\otimes e_s)(e_j^* \otimes e_t^* )v \, (E_{ij} \otimes E_{st})  \right ] \otimes ww^*
 \\[3pt] & = \left [ \sum  \overline{v_{i,s}} v_{j,t}(E_{ij} \otimes E_{st})  \right ]\otimes  ww^* 
 \\[3pt] & = \left [ \sum \overline{v_{i,s}} v_{j,t}(e_i\otimes e_s)(e_j \otimes e_t)^*  \right ]\otimes  ww^* 
 \\[3pt] & =  \left [ (\sum \overline{v_{i,s}} \,  e_i \otimes e_s ) \, (\sum \overline{v_{j,t}} \,  e_j\otimes e_t)^* 
 \right ]\otimes ww^*
 \\[3pt] & = uu^* \otimes ww^*,
\end{split}
\]
where  $u=\sum \overline{v_{a,b}} e_a \otimes e_b.$

Assuming the \jcp product $\Phi:\Mnf{n}\times \Mnf{n}\to \Mnf{m}$ has the form 
$\Phi(A,B) =\sum V_j^* (A\otimes B)V_j$ for rank one $V_j$ and setting
$\Psi_j= V_j^*(A\otimes B)V_j,$ it follows that
\[
 \Phi =\sum \Psi_j.
\]
 Since $V_j$ is rank one, there exist $v_j\in \FF^n\otimes\FF^n$ and $w_j\in \FF^m$ such that
 $V_j=v_j w_j^*.$ From what has already been proved, with 
\[
 v_j =\sum_{a,b=0}^{n-1} v_{j,a,b} \, e_a\otimes e_b,
\]
 setting 
\[
   u_j =\sum_{a,b=0}^{n-1} \overline{v_{j,a,b}}\,  e_a\otimes e_b,
\]
 gives
\[
 C_{\Psi_j} = u_j u_j^* \otimes w_j w_j^*.
\]
 Thus $C_\Phi =\sum C_{\Psi_j}$ has the desired form.

To prove the converse, suppose $\Psi:\Mnf{n}\times \Mnf{n}\to \Mnf{m}$ is
a \jcp product with Choi matrix
\[
 C_\Psi = uu^* \otimes ww^*  
\]
 for some   $u\in \FF^n\otimes \FF^n$ and $w\in \FF^m.$  Thus
\[
\begin{split}
 \Psi(E_{ij},E_{st}) &  = \left [(e_i \otimes e_s)^* uu^* (e_j\otimes e_t) \right ] \, \otimes ww^*
 \\[3pt] & =  u_{i,s} \overline{u_{j,t}} \, ww^*  = \overline{v_{i,s}} v_{j,t} \, ww^*
 \\[3pt] & =  \left[v^* (E_{ij}\otimes E_{st}) v\right ] \, ww^*.
 \end{split}
\]
 Hence, by bilinearity, 
 \[  
 \Psi(A,B)= \left [ v^*(A\otimes B)v  \right ] \, \otimes ww^* =  V^* (A\otimes B)V,
 \]
  where $V$ is the rank one matrix $V=v w^*.$ 
 From here a routine argument completes the proof.
\end{proof}

Proposition \ref{prop:SeparableKrausForm} has several immediate consequences regarding units of separable \jcp products. 
In particular, with  $\FF=\RR,$ the case of a real separable product,
we obtain the following corollary.
 
\begin{corollary}
    If  $\Phi:\Mnr{n} \times \Mnr{n} \to \Mnr{m}$ is a real separable \jcp product, 
    then $\Phi(A,B)$ is symmetric for all $A,B$. As a consequence, if $n \geq 2$, then a real separable \jcp product cannot have a left or right unit.
\end{corollary}

\begin{proof}
    The fact that $\Phi(A,B)$ is symmetric is immediate from equation \eqref{eq:SeparableKrausForm}, which shows that $\Phi(A,B)$ is an $\RR$-linear combination of symmetric matrices.  On the other hand, if $\unitL$ is a left unit of $\Phi$, then $\Phi(\unitL,B) = B$ is not symmetric if $B$ is not symmetric, contradicting the fact that $\Phi(A,B)$ is always symmetric
    (when $n\ge 2$). Right units are handled analogously. 
\end{proof}

\begin{corollary}
 \label{c:id-then-not-sep}
     Suppose $\Phi:\Mnc{n} \times \Mnc{n} \to \Mnc{n}$ is a separable \jcp product.
    
 \begin{enumerate}[(i)]\itemsep=6pt
  \item If  $n \geq 2,$  then $\Phi$ cannot have a positive semidefinite left or right unit;
  \item If the separable rank of $\Phi$ is less than $n^2$, then $\Phi$ cannot have either a left or right unit. 
  \end{enumerate}
\end{corollary}

\begin{proof}
    Suppose $\Phi:\Mnc{n} \times \Mnc{n} \to \Mnc{n}$ is separable. Using Proposition \ref{prop:SeparableKrausForm}, there exist collections of vectors $\{v_j\}_{j=1}^k \subset \mathbb{C}^n$ and $\{w_j\}_{j=1}^k \subset \mathbb{C}^{n^2}$ such that
 \begin{equation}
        \label{e:spann}
    \Phi(A,B) =  \sum_{j=1}^k \left [ v_j^*(A\otimes B)v_j  \right ]  w_jw_j^* \in \spann(\{w_j w_j^* \mid 1\le j\le k\}).
 \end{equation}
     Now, if $\Phi$ has a positive semidefinite left unit $\unitL \succeq 0$, then $\left( v_j^* (\unitL \otimes B) v_j\right) \geq 0$ for all positive semidefinite matrices $B$,  leading to the contradiction that  each psd matrix $B\in \Mnc{n}$
      is a finite
      conic linear combination of the finite set of psd matrices $\{ w_j w_j^*\}.$
     Hence no such $\unitL$ exists if $n\ge 2.$

     To prove the second statement, note that equation~\eqref{e:spann} implies
     the dimension of the range of $\Phi$ is at most $k.$ On the other hand,
     if $\Phi$ has an unit, the the dimension of the range of $\Phi$ is $n^2.$
\end{proof}

\begin{corollary}
    The Schur product and \JP are not separable for $n \geq 2$. 
\end{corollary}
\begin{proof}
Both the Schur and \JP have rank one psd units:
letting $e \in \mathbb{C}^n$ denote the vector of all ones, $ee^*$ is the unit of the Schur product. Similarly, $e_1 e_1^*$ is the is the unit of the \JP. Thus neither product is 
separable as a consequence of Corollary \ref{c:id-then-not-sep}. 
\end{proof}

\begin{remark}
    There exists a complex separable \jcp product $\Phi:\Mnc{2} \times \Mnc{2} \to \Mnc{2}$ that has separable rank $4$ and has a unit. Furthermore, $\Phi$ can be chosen so that $\Phi$ maps real matrices to real matrices. This example was generated with the aid of ChatGPT and was verified in Mathematica. A public Mathematica notebook that verifies the example is available at \href{https://github.com/NCAlgebra/UserNCNotebooks/tree/master/EvertJuryMcCullough}{https://github.com/NCAlgebra/UserNCNotebooks}.
 \qed
\end{remark}

\section{Numerical Experiments}
\label{sec:NumExp}

We now present numerical experiments designed to illustrate the quality of the lower bounds for the determinant of the convolution product obtained in Section \ref{sec:LowerBounds}. The numerical experiments in this section were run in Mathematica. A public Mathematica notebook that can be used to reproduce the experiments is available at the location \href{https://github.com/NCAlgebra/UserNCNotebooks/tree/master/EvertJuryMcCullough}{https://github.com/NCAlgebra/UserNCNotebooks}.

As our lower bounds are valid only for positive semidefinite inputs, we sample from the Wishart distribution $W_n(I_n/k,k)$ on $n \times n$ positive semidefinite matrices. Here $k$ is the number of degrees of freedom. Note that the choice of scale matrix $I_n/k$ gives the expected value $\mathbb{E}(X) = I_n$  for $X \sim W_n(I_n/k,k)$. As $k$ increases, the conditioning of the matrices drawn from $W_n(I_n/k,k)$ improves on average. Intuitively, when $k$ is large, a random matrix from $W_n(I_n/k,k)$ is closer to the identity matrix. 

Our experiment is performed as follows. For each fixed $n=2,\dots,20$ and $k= n,2n,4n,8n$, we randomly generate $1000$ pairs of matrices $P,Q$ sampled independently from the Wishart distribution $W_n(I_n/k,k)$. For each pair $P,Q$, we compute the determinant of the convolution product $\det(\CP(P,Q))$ as well as the lower bounds obtained from Proposition \ref{p:generic-lower}, Theorem \ref{t:lbJ}, and \cite[Theorem~2.1]{lower}. For each quantity, we record the median value over the 1000 random matrix pairs $(P,Q)$. We use the notation $L_{\mathrm{\jcp,gen}}$, $L_{\mathrm{\jcp,conv}}$, and $L_{\mathrm{GMV}}$ to denote the medians of the lower bound obtained from Proposition \ref{p:generic-lower}, Theorem \ref{t:lbJ}, and \cite[Theorem~2.1]{lower}, respectively. The results are presented in Figure \ref{fig:NumExp}.

\begin{figure}[htbp]
    \centering
    \resizebox{1 \textwidth}{!}{\input{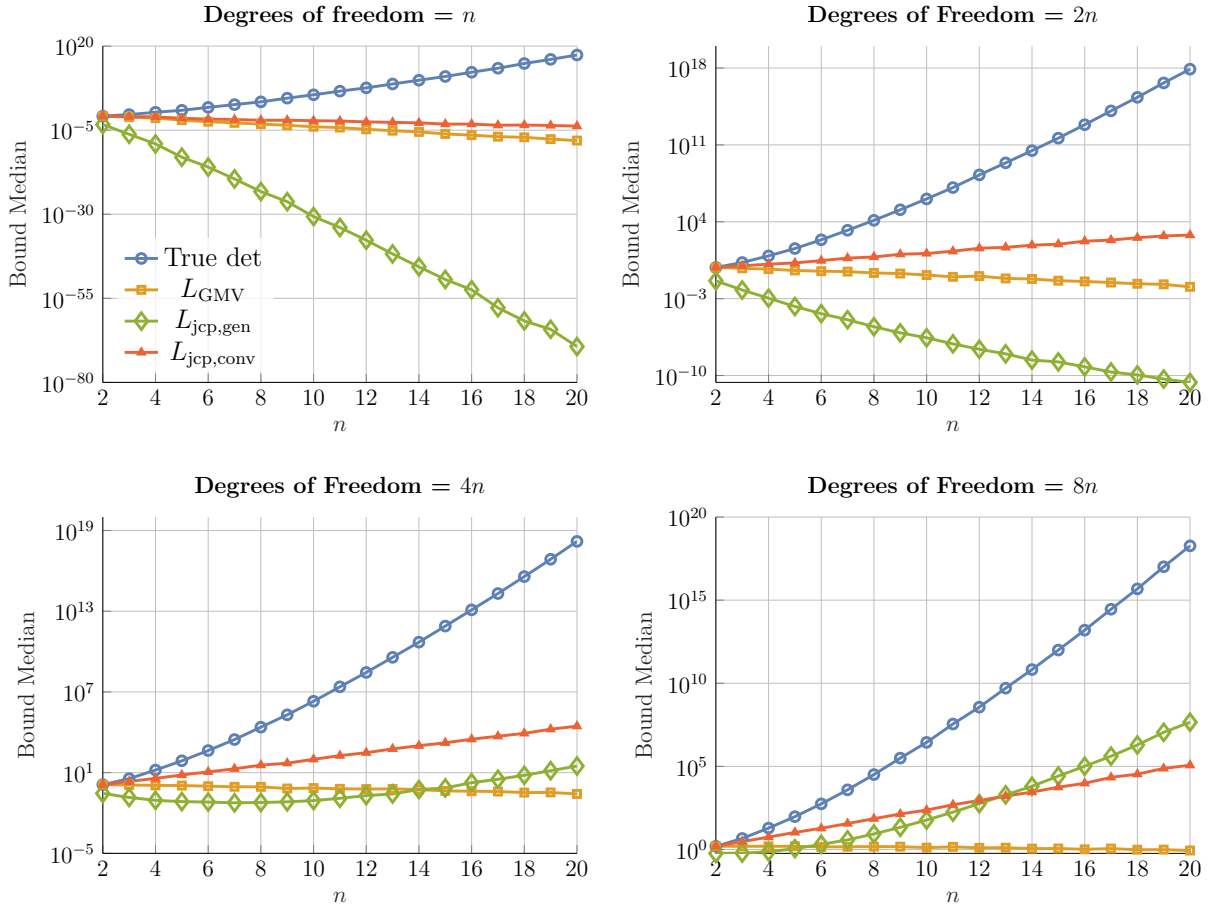}}
    \caption{Medians of determinant lower bounds for $n \times n$ Wisharts.}
    \label{fig:NumExp}
\end{figure}

 As expected from Remark \ref{r:lbJ}, we see for all choices of $k$ that $\Lcpc \geq \Lgmv$. We observe that the gap between these medians increases as $k$ increases. This is perhaps unsurprising, as the typical distance between $P$ and $Q$ decreases as $k$ increases, and one can compute that if $P=Q$, then the bound obtained from Theorem~\ref{t:lbJ} simplifies to $2^{n-1} p_{00}^n \det(P)$. On the other hand, for $P=Q$, the bound from \cite[Theorem~2.1]{lower} simplifies to $2p_{00}^n \det(P)$.

For $k=n$, we observe that $\Lcpg$ is significantly smaller than both $\Lcpc$ and $\Lgmv$. However, the relative strength of $\Lcpg$ increases strongly with $k$. When $k=4n$, we observe that $\Lcpg$ exceeds $\Lgmv$ for larger values of $n$. Additionally, the plot suggests that $\Lcpg$ may exceed $\Lcpc$ for sufficiently large values of $n$. When $k=8n$, we see that $\Lcpg$ is larger than both $\Lcpc$ and $\Lgmv$ for all $n \geq 14$. This is also perhaps expected. As mentioned above, $P$ and $Q$ are closer to the identity matrix (in expectation) as $k$ increases. One may verify that if $P=Q=I_n$, then the bound obtained from Proposition~\ref{p:generic-lower} is $n!$, which is equal to the true determinant $\det(\CP(I_n,I_n))$. On the other hand, for $P=Q=I_n$ the bounds obtained from Theorem~\ref{t:lbJ} and \cite[Theorem~2.1]{lower} are $2^{n-1}$ and $2$, respectively.


\begin{thebibliography}{99}

\bibitem[AM]{AM}
  J. Agler and J. E. McCarthy, {\it Pick Interpolation and Hilbert Function Spaces,}
  Graduate Studies in Mathematics, 44, Amer. Math. Soc., Providence, RI, 2002.

  \bibitem[Alf]{Alf}
  E. M. Alfsen, {\it Compact Convex Sets and Boundary Integrals,}
  2nd ed., Ergebnisse der Mathematik und ihrer Grenzgebiete, 57,
  Springer-Verlag, New York, 1971.

  

\bibitem[Arv]{Arv}
  W. Arveson, {\it An Invitation to $C^*$-Algebras,}
  Graduate Texts in Mathematics, 39, Springer-Verlag, New York, 1976.

\bibitem[BP]{BP}
  A. Berman and R. J. Plemmons, {\it Nonnegative Matrices in the Mathematical Sciences,}
  Classics in Applied Mathematics, 9, SIAM, Philadelphia, PA, 1994.

 \bibitem[Choi]{Choi}
  M.-D. Choi, {\it Completely Positive Linear Maps on Complex Matrices,} Linear Algebra Appl., 10 (1975), pp. 285--290.

  \bibitem[Dav]{Dav}
  K. R. Davidson, {\it $C^*$-Algebras by Example,}
  Fields Institute Monographs, 6, Amer. Math. Soc., Providence, RI, 1996.

  \bibitem[DMM]{DMM}
  M. A. Dritschel, S. Marcantognini, and S. McCullough,
  {\it Interpolation in Semigroupoid Algebras,}
  J. Reine Angew. Math., 606 (2007), pp. 1--40.


\bibitem[FJ07]{FallatJohnson} S.M. Fallat and C.R. Johnson, {\it Hadamard Duals, Retractability and Oppenheim's Inequality}, Oper. Matrices, 1 (2007). DOI: \url{dx.doi.org/10.7153/oam-01-21}.  

  
 \bibitem[GMV+]{lower}
 Dominique Guillot, Javad Mashreghi and Prateek Kumar Vishwakarma, 
   {\it Oppenheim--Schur inequalities for causal products},
   arXiv:2602.21056v1
   
 \bibitem[GMV++]{dominiqueOppenheim}
 Dominique Guillot, Javad Mashreghi and Prateek Kumar Vishwakarma, 
  {\it Sharp lower bounds for generalized operator products}, arXiv:2601.00409v1

  \bibitem[Horn]{Horn}
  R. A. Horn, {\it The Hadamard Product,} in {\it Matrix Theory and Applications},
  C. R. Johnson, ed., Proc. Sympos. Appl. Math., 40 (1990), pp. 87--169.

\bibitem[HSR]{HSR}
  M. Horodecki, P. W. Shor, and M. B. Ruskai,
  {\it Entanglement Breaking Channels,}
  Rev. Math. Phys., 15 (2003), pp. 629--641.

  \bibitem[Jury]{Jury} M. T. Jury, {\it Matrix Products and Interpolation Problems in Hilbert Function Spaces,} Ph.D. thesis, Washington University in St. Louis, ProQuest LLC, Ann Arbor, MI, 2002.

\bibitem[KPTT]{KPTT} A. Kavruk, V. I. Paulsen, I. G. Todorov, and M. Tomforde, {\it Tensor products of operator systems}, J. Funct. Anal., 261 (2011), pp. 267--299.

  \bibitem[LS]{LS}
  D. D. Lee and H. S. Seung, {\it Learning the Parts of Objects by Non-negative Matrix Factorization,}
  Nature, 401 (1999), pp. 788--791.

  \bibitem[Opp]{Oppenheim} A. Oppenheim, {\it Inequalities connected with definite hermitian forms,} J. London Math. Soc., 5(2) (1930), pp. 114--119.

\bibitem[Pau]{PaulsenBook}
  V. I. Paulsen, \textit{Completely Bounded Maps and Operator Algebras}, Cambridge Studies in Advanced Mathematics, vol.\ 78, Cambridge University Press, Cambridge, 2002. ISBN 0-521-80888-6. DOI: \url{https://doi.org/10.1017/CBO9780511546462}. 

\bibitem[Sch]{Schur}
    Issai Schur, \textit{Bemerkungen zur Theorie der beschr\"ankten Bilinearformen mit unendlich vielen Ver\"anderlichen,} J. Reine Angew. Math., 140 (1911), pp. 1--28.

    \bibitem[Wagner]{Wagner} D. G. Wagner, {\it Total Positivity of Hadamard Products,} J. Math. Anal. Appl., 163 (1992), pp. 459--483.

    \bibitem[Watrous]{Watrous}
  J. Watrous, {\it The Theory of Quantum Information,}
  Cambridge University Press, Cambridge, 2018.
  
\end{thebibliography}
\end{document}